\documentclass[10pt, a4paper]{amsart}

\usepackage[T1]{fontenc}
\usepackage{microtype}

\usepackage{amsmath}								
\usepackage{amssymb}
\usepackage{amsthm}
\usepackage{amscd}
\usepackage{mathtools}
\usepackage{amsfonts}
\usepackage{stmaryrd}

\usepackage{extarrows}

\usepackage[bookmarks=true, colorlinks=true, linkcolor=blue!50!black,
citecolor=orange!50!black, urlcolor=orange!50!black, pdfencoding=unicode]{hyperref}
\usepackage{color}

\usepackage{tikz}									
\usetikzlibrary{matrix}
\usetikzlibrary{patterns}
\usetikzlibrary{matrix}
\usetikzlibrary{positioning}
\usetikzlibrary{decorations.pathmorphing}
\usetikzlibrary{cd}

\usepackage{fullpage}
\usepackage{enumerate}

\newtheorem*{theorem*}{Theorem}

\newtheorem{maintheorem}{Theorem}[section]

\newtheorem{theorem}{Theorem}[section]
\newtheorem{lemma}[theorem]{Lemma}
\newtheorem{proposition}[theorem]{Proposition}
\newtheorem{corollary}[theorem]{Corollary} 
\theoremstyle{definition}
\newtheorem{nota}[theorem]{Notation}
\newtheorem{definition}[theorem]{Definition}
\newtheorem{example}[theorem]{Example}

\newtheorem{remark}[theorem]{Remark}

\newtheoremstyle{myitemstyle}						
	{}			
	{}			
	{}			
	{}			
	{}			
	{.}			
	{ }			
	{}			
\theoremstyle{myitemstyle}
\newtheorem{myitemthm}{}

 \newcommand{\martincolor}[1]{{\color{green!60!black} \sf #1}}

 \newcommand{\annette}[1]{{\color{teal} \sf Annette: [#1]}}

 \newcommand{\andreascolor}[1]{{\color{purple!60!black} \sf #1}}

\newcommand{\R}{\mathbb{R}}

\newcommand{\Z}{\mathbb{Z}}

\newcommand{\Q}{\mathbb{Q}}
\newcommand{\C}{\mathbb{C}}

\newcommand{\G}{\mathbb{G}}

\newcommand{\T}{\mathbb{T}}

\newcommand{\NLargeLattice}[1]{N_{#1}^{\Z}}
\newcommand{\MLargeLattice}[1]{M_{#1}^{\Z}}
\newcommand{\LargeLattice}[1]{\Lambda_{#1}^{\Z}}
\newcommand{\SmallLattice}[1]{\Lambda_{#1}^{c}}

\newcommand{\calE}{\mathcal{E}}

\newcommand{\calS}{\mathcal{S}}

\newcommand{\mc}[1]{{\mathcal{#1}}}

\DeclareMathOperator{\Pic}{Pic}
\DeclareMathOperator{\Spec}{Spec}
\DeclareMathOperator{\Hom}{Hom}
\DeclareMathOperator{\Ext}{Ext}

\DeclareMathOperator{\Aut}{Aut}

\DeclareMathOperator{\calHom}{\mathcal{H}om}

\DeclareMathOperator{\id}{id}

\DeclareMathOperator{\GL}{GL}

\newcommand{\trop}{{\mathrm{trop}}}

\DeclareMathOperator{\Sym}{Sym}

\DeclareMathOperator{\Aff}{\mathrm{Aff}}

\DeclareMathOperator{\rk}{rk}
\DeclareMathOperator{\NS}{NS}
\DeclareMathOperator{\End}{End}

\newcommand\diag{{\mathrm{diag}}}
\newcommand\main{{\mathrm{main}}}

\newcommand{\an}{\mathrm{an}}

\title{Semi-homogeneous vector bundles on abelian varieties:  \\ moduli spaces and their tropicalization } 

\date{}

\author{Andreas Gross}
\address{Institut f\"ur Mathematik, Goethe--Universit\"at Frankfurt,
60325 Frankfurt am Main, Germany}
\email{gross@math.uni-frankfurt.de}

\author{Inder Kaur}
\address{School of Mathematics \& Statistics, University of Glasgow,
Glasgow G12 8QQ, UK}
\email{inder.kaur@glasgow.ac.uk}

\author{Martin Ulirsch}
\address{Institut f\"ur Mathematik, Universit\"at Paderborn,
33098 Paderborn, Germany}
\email{ulirsch@math.uni-paderborn.de}

\author{Annette Werner}
\address{Institut f\"ur Mathematik, Goethe--Universit\"at Frankfurt,
60325 Frankfurt am Main, Germany}
\email{werner@math.uni-frankfurt.de}

\begin{document}

\maketitle

\begin{abstract} 
Let $A$ be an abelian variety with  totally degenerate reduction over a non-Archimedean field. We describe the moduli space of semihomogeneous vector bundles on $A$ from the perspective of non-Archimedean uniformization and show that the essential skeleton may be identified with a tropical analogue of this moduli space.  For $H=0$ our moduli space may be identified with the moduli space $M_{0,r}(A)$ of semistable vector bundles with vanishing Chern classes on $A$. In this case we construct a surjective analytic morphism from the character variety of the analytic fundamental group of $A$ onto $M_{0,r}(A)$, which naturally tropicalizes.  One may view this construction as a non-Archimedean uniformization of $M_{0,r}(A)$. 
\end{abstract}

\setcounter{tocdepth}{1}
\tableofcontents


\section*{Introduction}

In this article we study semi-homogeneous vector bundles on abelian varieties (in the sense of Mukai \cite{Mukai}) and their moduli spaces through the lens of tropical geometry (expanding on \cite{GrossUlirschZakharov}) and the non-Archimedean approach to the SYZ fibration via essential skeletons of Berkovich analytic spaces (introduced in \cite{KontsevichSoibelman}). 

Throughout we will work over an algebraically closed field $K$ of characteristic zero. From Section \ref{section_AppellHumbert} onwards we also assume that $K$ is complete with respect to a non-trivial non-Archimedean absolute value.

\subsection*{Semi-homogeneous vector bundles and their moduli spaces} Let $A$ be an abelian variety over an algebraically closed field $K$ of characteristic $0$. The engine that drives this whole article is the theory of (semi-)homogeneous vector bundles on $A$ that has been developed in \cite{Miyanishi} and \cite{Mukai} (also see \cite{Morimoto} and \cite{Matsushima} for a complex analytic predecessor of this notion).  Write $T_x\colon A\rightarrow A$ for the translation by a closed point $x$ of $A$. A vector bundle $E$ on $A$ is said to be \emph{semi-homogeneous} if we have $T_x^\ast E\simeq L\otimes E$, for a suitable line bundle $L$ on $A$ (possibly depending on $x$) and \emph{homogeneous} if $L$ may be chosen to be trivial.  
 By \cite[Theorem 5.8]{Mukai} a simple vector bundle $E$ on an abelian variety $A$ is semi-homogeneous if and only if there exists an isogeny $f\colon B\to A$ and a line bundle $L$ on $B$ such that $E\cong f_*(L)$.

For a vector bundle $E$ on $A$, its slope $\delta(E)$ is the class $\det(E)/r(E)$ of its determinant divided by its rank in the Néron-Severi group $\NS(A)_\Q$ with coefficients in $\Q$. By \cite{Mukai}, given a class $H\in \NS(A)_\Q$, there exists a simple semi-homogeneous vector bundle $E$ on $H$ with $\delta(E)=H$, and any two such bundles coincide up to tensoring with a line bundle in $\Pic^0(A)$ (for details see Theorem \ref{thm:existence of simple semi-homogeneous bundles} below). We denote by $n(H)$ the rank of any such simple semi-homogeneous vector bundle.
Let $k\in \Z_{>0}$ and set $r=k\cdot n(H)$. Denote by $M_{H,k}(A)$ the normalization of the locus of semi-homogeneous vector bundles $E$ of slope $H$ and rank $r=k\cdot n(H)$ in the moduli space semistable torsion free sheaves on $A$. 
We shall see in Theorem \ref{thm:description of moduli spaces} below that its image in the moduli space defines a connected component of the moduli space. 

We write 
\[
\Sigma(H)= \{x\in A^\vee \mid P_x\otimes E \cong E\} \ ,
\]
where $P_x$ is the line bundle on $A$ corresponding to $x$ in the the dual abelian variety $A^\vee$ and $E$ is any semi-homogeneous vector bundle with $\delta(E)=H$. Let $\pi^\ast: A^\vee  \rightarrow A^\vee / \Sigma(H)$ be the isogeny with kernel $\Sigma(H)$, and let $\pi: A_H \rightarrow A$ be the dual isogeny, where $A_H$ denotes the dual of $A^\vee / \Sigma(H)$.


\begin{maintheorem} [see Theorem \ref{thm:description of moduli spaces}]\label{mainthm_SHmoduli}
Let $A$ be an abelian variety over an algebraically closed field $K$ of characteristic zero and fix a class $H\in \NS(A)_{\Q}$.
\begin{enumerate}[(i)]
\item When $k=1$, then 
\begin{equation*}
M_{H,1}(A) \cong \Pic^{\pi^*H}(A_H) 
\end{equation*}
\item When $k \geq 1$, then there is  a canonical isomorphism 
\begin{equation*}
    \Sym^k M_{H,1}(A)\xlongrightarrow{\sim} M_{H,k}(A) \ .
\end{equation*}
\end{enumerate}
\end{maintheorem}


By \cite[Theorem 2 and the discussion on page 2]{MehtaNori} a vector bundle $E$ on $A$ is semi-homogeneous if and only it is \emph{semistable with projective Chern class zero}, meaning that the total Chern class is given by $\left(1+\frac{c_1(E)}{r}\right)^r$ where $r=k\cdot n(H)$ is the rank of $E$ and $E$ is semistable with respect to some polarization on $A$. So the component $M_{H,k}(A)$ may also be reinterpreted as a moduli space of semistable bundles with projective Chern class zero (and fixed slope $H$). 

If $H=0$, $M_{H,k}(A)$ the moduli space of semistable vector bundles of rank $r=k$, all of whose Chern classes are trivial (which is why we often switch the index from $k$ to $r$). This space appears e.g.\ in \cite{SimpsonII} as the locus in the Dolbeault moduli space of topologically trivial semistable Higgs bundles where the Higgs field is zero. In this case, it appears to be well-known in the community that the morphism in Part (ii) of Theorem \ref{mainthm_SHmoduli} induces an isomorphism onto its image in the moduli space of semistable vector bundles. This fact can, for example, also be deduced from the description of the moduli space of $\Lambda$-modules on $A$ as a symmetric power in \cite[Theorem 3.14 and Proposition 3.16]{FrancoTartella} (also see \cite[Proposition 1.6]{BKU}).

Theorem \ref{mainthm_SHmoduli} is a moduli-theoretic reinterpretation of the structure results on semi-homogeneous vector bundles in \cite{Mukai}. The technique used in \cite{FrancoTartella}, as well as the technique used in the proof of Theorem \ref{mainthm_SHmoduli}, is based on Mukai's work on Fourier-Mukai transforms \cite{mukai_1981}, where he already observed the equivalence of the category of homogeneous vector bundles on an abelian variety and that of finite-length sheaves on its dual.

Part (i) of Theorem \ref{mainthm_SHmoduli} is a slightly more explicit version of \cite[Proposition 3.1]{Gulbrandsen}. In dimension one, Theorem \ref{mainthm_SHmoduli} specializes to Atiyah's classification of vector bundles on elliptic curves in \cite{Atiyah_elliptic}, which has found a moduli-theoretic interpretation in \cite{Tu_modulibundlesellipticcurve} and later in \cite{FourierMukaiGenus1} from a Fourier-Mukai perspective. Moduli spaces of simple semi-homogeneous sheaves also play a crucial role in Lane's study of Fourier-Mukai equivalences between twisted derived categories in \cite{Lane_semihomogeneoustwisted}.

\subsection*{Essential skeletons and tropicalization}
From now on we assume that $K$ (in addition to being algebraically closed and of characteristic zero) is also complete with respect to a non-trivial non-Archimedean absolute value. Let $A$ be an abelian variety over $K$ with totally degenerate reduction. By non-Archimedean uniformization, the Berkovich analytic space $A^{\an}$ is a quotient $T^{\an}/\Lambda$ of an algebraic torus $T\simeq \G_m^g$ by a lattice $\Lambda\simeq \Z^g$. The non-Archimedean skeleton of $A^{\an}$ in the sense of Berkovich \cite{Berkbook} is naturally identified with a real torus $A^{\trop}=N_\R/\Lambda$, where $N_\R$ is the $\R$-linear space spanned by the cocharacter lattice $N$ of $T$; we refer to $A^{\trop}$ as the \emph{tropicalization} of $A$.

Berkovich skeletons are usually not unique. In \cite{KontsevichSoibelman}, in an effort to construct a non-Archimedean analogue of mirror symmetry and, in particular, the SYZ-fibration, Kontsevich and Soibelman suggest the construction of an \emph{essential skeleton} of the Berkovich analytic space associated to a Calabi--Yau variety (also see \cite{MustataNicaise, NicaiseXu, NicaiseXuYu} for further developments). Since the moduli spaces $M_{H,k}$ are all Calabi--Yau varieties, they do admit a canonical strong deformation retraction $\tau$ onto their essential skeleton $\Sigma(M_{H,k})$.

In Section \ref{section_tropicalvectorbundles} below, we develop a tropical analogue of the geometry of semi-homogeneous vector bundles on $A^{\trop}$, expanding on the results in \cite{GrossUlirschZakharov} on metric graphs (and inspired by the case of line bundles developed in \cite{MikhalkinZharkov}). In particular, we prove a tropical analogue of Mukai's classification  of semi-homogeneous bundles \cite{Mukai}. Mukai's classification does not immediately allow us to construct a well-defined tropicalization map, whose target is the na\"ive tropical analogue of the moduli space $M_{H,k}(A)$. Given a fixed class $H\in \NS(A)_{\Q}$, we identify a natural smallest sublattice $\SmallLattice{H}$, that allows to construct a refined moduli space $M_{H,k}^{\SmallLattice{H}}(A^{\trop})$ of \emph{equivalence classes} of semi-homogeneous tropical vector bundles, that is the target of a natural tropicalization map $\trop\colon M_{H,k}(A)^{\an}\rightarrow M_{H,k}^{\SmallLattice{H}}(A^{\trop})$ (see Definition \ref{smalllattice{H}} for the precise definition of $\SmallLattice{H}$ and Definition \ref{def:equivalence of vector bundles} for the definition of $M_{H,k}^{\SmallLattice{H}}(A^{\trop})$).

The following Theorem \ref{mainthm_skeleton=tropicalization} tells us that this construction is precisely the retraction to the essential skeleton of $M_{H,k}(A)^{\an}$.

\begin{maintheorem}[see Theorem \ref{thm:skeleton in general case}]\label{mainthm_skeleton=tropicalization}
Let $A$ be an abelian variety over an algebraically closed complete non-Archimedean field $K$ with totally degenerate reduction, and fix $H\in \NS(A)_{\Q}$. Then there is a natural isomorphism
\begin{equation*}
    M_{H,k}^{\SmallLattice{H}}(A^{\trop})\xlongrightarrow{\sim} \Sigma\big(M_{H,k}(A)\big)
\end{equation*}
that makes the diagram
\begin{equation*}
    \begin{tikzcd}
       & M_{H,k}(A)^{\an} \arrow[rd,"\tau"] \arrow[ld,"\trop"']&\\
       M_{H,k}^{\SmallLattice{H}}(A^{\trop}) \arrow[rr,"\sim"] && \Sigma\big(M_{H,k}(A)\big)
    \end{tikzcd}
\end{equation*}
commute. 
\end{maintheorem}

In the case $H=0$ (and thus $r=k$) the moduli space $M_{0,r}(A)$ is isomorphic to the moduli space of semistable vector bundles of rank $r$ on $A$ with vanishing Chern classes. In this case, the moduli space $M_{0,r}^{^{\SmallLattice{H}}}(A^{\trop})$ may be identified with the \emph{main component} $M_{0,r}^{\main}(A^{\trop})$ of a (naively defined) moduli space $M_{0,r}(A^{\trop})$ of homogeneous tropical vector bundles of rank $r$ and with trivial Neron-Severi class (see Section \ref{section_M0r} below for details). So Theorem \ref{mainthm_skeleton=tropicalization}, in particular, describes the essential skeleton of the moduli space of semistable vector bundles with vanishing Chern classes.

Theorem \ref{mainthm_skeleton=tropicalization} is a generalization of \cite[Theorem D]{GrossUlirschZakharov}, which proves an analogous result on Tate curves (the case $g=1$).  
It forms another addition to the dictionary between non-Archimedean skeletons of algebraic moduli spaces and their tropical analogues. The pioneering works in this direction are \cite{BakerRabinoff}, which deals with the  Jacobian, and \cite{ACP}, which deals with the moduli space of curves.

\subsection*{Representations, homogeneous bundles, and tropicalization} Homogeneous bundles play an important role in the representation theory of the fundamental group of an abelian variety. 
For an abelian variety $A$ with totally degenerate reduction, we write $A^{\an}=T^{\an}/\Lambda$ for an algebraic torus $T=\G_m^g$, with $\Lambda$ the analytic fundamental group of $T^{an}$. 
 
 We see in Section \ref{section_charactervariety} below that for semisimple representations the characterization of line bundles on $A$ due to Gerritzen \cite{Gerritzen}, gives rise to
a natural surjective analytic morphism
\begin{equation*}
\eta_A^{\an}\colon X_r(\Lambda)^{\an}\longrightarrow M_{0,r}(A)^{\an}
\end{equation*}
from the (analytification of the) \emph{character variety} 
\begin{equation*}
X_r(\Lambda)=\Hom(\Lambda,\GL_r)\sslash{\GL_r}
\end{equation*}
of $\Lambda$ to the (analytification of the) moduli space $M_{0,r}(A)$ of homogeneous bundles of rank $r$ on $A$ (or equivalently the moduli space of semistable bundles with vanishing Chern classes). 

In Sections \ref{section_tropcharvar} and \ref{section_tropcharvarII} below we construct a tropical analogue $X^{\trop}_r(\Lambda)$ of the character variety $X_r(\Lambda)$ as well as a natural tropicalization map $\trop\colon X_r(\Lambda)^{\an}\rightarrow X^{\trop}_r(\Lambda)$ that exhibits the main component $X^{\trop}_r(\Lambda)^{\diag}$ of $X^{\trop}_r(\Lambda)$ (parametrizing diagonalizable representations) as the essential skeleton $\Sigma\big(X_r(\Lambda)\big)$ of the open Calabi-Yau variety $X_r(\Lambda)$. In other words, we find an isomorphism $X^{\trop}_r(\Lambda)^{\diag}\xrightarrow{\sim}\Sigma\big(X_r(\Lambda)\big)$ that identifies the tropicalization map $\trop$ with the retraction $\tau$ to $\Sigma\big(X_r(\Lambda)\big)$.

The following Theorem \ref{mainthm_troprep} provides us with a tropical analogue of $\eta_A^{\an}$ that is compatible with the process of tropicalization.

\begin{maintheorem}[see Corollary \ref{cor_tropicalcorrespondence} and Proposition \ref{prop_compatibilityrep->bundle}]\label{mainthm_troprep}
There is a natural surjective morphism  
\begin{equation*}\eta_{A}^{\trop}\colon 
X^{\trop}_r(\Lambda)\longrightarrow M_{0,r}(A^{\trop})
\end{equation*}
(of integral affine orbifolds) that is given by associating to every tropical representation $\rho\colon\Lambda\rightarrow \GL_n(\T)$ 
 a  tropical homogeneous vector bundle $E(\rho)$ and which naturally agrees with $\eta_A^{\an}$ on the main components.  %
\end{maintheorem}

In other words, we may summarize our results by saying that the natural diagram
\begin{equation*}\begin{tikzcd}
X_r(\Lambda)^\an \arrow[rrrr,"\eta_A^{\an}"]\arrow[rd,"\tau"']\arrow[rdd,bend right,"\trop"'] &&&& M_{0,r}(A)^{\an}\arrow[ld,"\tau"]\arrow[ldd,bend left,"\trop"]\\
& \Sigma\big(X_r(\Lambda)\big) \arrow[rr,"\eta_A^{\an}"] && \Sigma\big(M_{0,r}(A)\big) &\\
& X^{\trop}_r(\Lambda)^{\diag} \arrow[rr,"\eta_{A}^{\trop}"]\arrow[u,"\sim"] \arrow[d,"\subseteq"']&& M_{0,r}^{\main}(A^{\trop})\arrow[u,"\sim"']\arrow[d,"\subseteq"] &\\
& X_r^{\trop}(\Lambda) \arrow[rr,"\eta_{A}^{\trop}"] && M_{0,r}(A^{\trop}) &
\end{tikzcd}\end{equation*}
commutes.

 A $p$-adic analogue of the Corlette-Simpson correspondence on abeloid varieties was proven in \cite{HeuerMannWerner}. In this $p$-adic setting, on the representation side continuous representations of the (pro-finite) \'etale  fundamental group are considered. Nevertheless,  working with only ``half'' of the fundamental group in this article provides us with a connection to tropical geometry.

\subsection*{Acknowledgements} The authors would like to thank Luca Battistella, Barbara Bolognese, Michel Brion, Francesca Carocci, Camilla Felisetti, Emilio Franco, Johannes Horn, Katharina H\"ubner, Arne Kuhrs, Alex K\"uronya, Chunyi Li, Mirko Mauri, Dhruv Ranganathan, and Dmitry Zakharov for helpful conversations and discussion en route to this article. 
We also thank the referee of this article for many useful comments. Particular thanks are due to Jakob Stix, who helped us clarify the role of the moduli space in Section \ref{subsec:results on semi-homogeneous bundles} and Tyler Lane, who spotted an inaccuracy in an earlier version of Theorem \ref{mainthm_SHmoduli} and explained his insights in several subsequent discussions.


This project has received funding from the Deutsche Forschungsgemeinschaft (DFG, German Research Foundation) TRR 326 \emph{Geometry and Arithmetic of Uniformized Structures}, project number 444845124, and TRR 358 \emph{Integral Structures in Geometry and Representation Theory}, project number 491392403, as well as from the Deutsche Forschungsgemeinschaft (DFG, German Research Foundation) Sachbeihilfe \emph{From Riemann surfaces to tropical curves (and back again)}, project number 456557832, and the DFG Sachbeihilfe \emph{Rethinking tropical linear algebra: Buildings, bimatroids, and applications}, project number 539867663, within the  
SPP 2458 \emph{Combinatorial Synergies}, as well as from the LOEWE grant \emph{Uniformized Structures in Algebra and Geometry} and from the Marie-Sk\l{}odowska-Curie-Stipendium Hessen (as part of the HESSEN HORIZON initiative).


\section{Semi-homogeneous bundles and their moduli spaces}
\label{subsec:results on semi-homogeneous bundles}


\vspace{0.2 cm}
In this section we let $K$ be an algebraically closed field of characteristic zero.
     We recall the basic definitions and results related to semi-homogeneous vector bundles on abelian varieties. Throughout this section, (semi)-stability refers to Gieseker (semi)-stability.

\subsection{Semi-homogeneous vector bundles on abelian varieties}

Let $A$ be an abelian variety of dimension $g$  over 
$K$ and write $T_x\colon A\rightarrow A$ for the translation $a\mapsto x+a$ by an element $x\in A(K)$.

\begin{definition}
A vector bundle $E$ on $A$ is called \emph{homogeneous}, if $T_x^\ast E\simeq E$ for every $x\in A(K)$. Line bundles are homogeneous if and only if they are algebraically equivalent to zero, or, equivalently, if their class in the Néron--Severi group $\NS(A)= \Pic(A)/\Pic^0(A)$ is trivial. 

A vector bundle $E$ is called \emph{semi-homogeneous} if for every $x \in A(K)$, there exists a line bundle $L$ (possibly depending on $x$) on $A$ such that $T_x^{*}E = E \otimes L$. 
\end{definition}

 

 

\begin{definition}  Let $\NS(A)_\Q$ denote the Néron--Severi group with rational coefficients.
For a vector bundle $E$ on $A$ we denote its  \emph{slope} by 
\[
\delta(E)=\frac{[\det(E)]}{r(E)}
\]
in $\NS(A)_\Q$. 
By $[\det]$, we mean the class of the determinant in the Néron--Severi group. Therefore fixing the slope for us implies that we are not fixing the determinant but its class. 
\end{definition}

\begin{remark}
Recall that a vector bundle $E$ is \emph{simple} if the only endomorphisms of $E$ are the scalars i.e.\ $\End(E) = K$. By {\cite[Proposition 6.16]{Mukai}} every simple semi-homogeneous bundle is stable.
\end{remark}

The following theorem provides a classification of semi-homogeneous vector bundles on abelian varieties and generalises the classification of vector bundles on elliptic curves given by Atiyah in \cite{Atiyah_elliptic}.

\begin{theorem}[{\cite[Proposition 6.2, Proposition 6.18]{Mukai}}]
\label{thm:decomposition of semi-homogeneous vector bundles}
 Let $E$ be a semi-homogeneous vector bundle. Then there exist simple semi-homogeneous vector bundles $E_1,\ldots , E_k$ with $\delta(E_i)=\delta(E)$ and for every $1\leq i\leq k$ a unipotent vector bundle $U_i$ such that
 \[
 E\cong \bigoplus_{i=1}^k U_i\otimes E_i \ .
 \]
\end{theorem}

\begin{definition}
    Let $H\in \NS(A)_\Q$. We define 
    $\calS_{H}$ to be the set of all isomorphism classes of simple semi-homogeneous vector bundles $E$ on $A$  with $\delta(E)=H$.
\end{definition}

The following result gives the non-emptiness of $\calS_H$.

\begin{theorem}[{\cite[Proposition 6.17, Corollary 6.23, Theorem 5.8]{Mukai}}]
\label{thm:existence of simple semi-homogeneous bundles}
 Let $H\in \NS(A)_\Q$. Then there exists at least one simple semi-homogeneous bundle $E$ on $A$ with $\delta(E)=H$; in other words $ \calS_H\neq \emptyset$. If $E,E'\in \calS_H$, then there exists a line bundle $M\in \Pic^0(A)$ with $E\otimes M\cong E'$. Moreover, for every $E\in \mc S_H$, there exists an isogeny $f\colon A'\to A$ and a line bundle $L$ on $A'$ with $f_*L\cong E$.
\end{theorem}



\begin{nota}
For a semi-homogeneous vector bundle $E$ on $A$ we denote
\[K(E)= \{x \in A\mid T_x^*E\cong E\} \]

and 
\[
\Sigma(E)= \{x\in A^\vee \mid P_x\otimes E \cong E\} \ ,
\]
where $P_x$ is the line bundle on $A$ corresponding to $x$ in the the dual abelian variety $A^\vee$.

Given $H\in \NS(A)_\Q$, we denote $n(H)=r(E)$ and $\Sigma(H)=\Sigma(E)$ for some simple semi-homogeneous vector bundle $E$ with $\delta(E)=H$. By Theorem \ref{thm:existence of simple semi-homogeneous bundles}, both $n(H)$ and $\Sigma(H)$ are well-defined, since any two choices of $E$ coincide up to tensoring with a line bundle in $\Pic^0(A)$. 
\end{nota}

By \cite[Proposition 3.2]{Mukai}, $K(E)$ is a closed subgroup scheme of $A$, and $\Sigma(E)$ is a closed subgroup scheme of  $A^\vee$. In particular both are reduced.

\begin{theorem}[{\cite[Theorem 1.2]{Oda}, \cite[Lemma 3.13]{Mukai}}]
\label{thm:condition on push-forward to be simple}
    Let $f\colon B\to A$ be an isogeny of abelian varieties and let $L$ be a line bundle on $B$. Then $f_*(L)$ is simple if and only if $K(L)\cap \ker(f)=0$. 
\end{theorem}

\begin{lemma}
\label{lem:slope of push-forward}
Let $f\colon B\to A$ be an isogeny of abelian varieties and let $L$ be a line bundle on $B$. Then we have $f^*\delta(f_*L)=L$ in $\NS(B)_\Q$.
\end{lemma}

\begin{proof}
Since pull-backs commute with determinants we have
\[
f^*\delta(f_*L)=\delta(f^*f_*L) \ .
\]
The isogeny $f$ is a $\ker(f)$-torsor.
Therefore,
\[
f^*f_*L\cong \bigoplus_{\lambda\in \ker(f)} T_\lambda^*L  \ .
\]
As translates of $L$ are algebraically equivalent, the assertion follows.
\end{proof}

Fix a class $H\in\NS(A)_\Q$.
By Theorem \ref{thm:existence of simple semi-homogeneous bundles} there exists a simple semi-homogeneous bundle $E$ of slope $H$, which is of the form $E\cong f_*L$ for an isogeny $f\colon A'\to A$ and a line bundle $L$ on $A'$. The simple semi-homogeneous bundles of slope $H$ are then precisely the bundles of the form $E\otimes M \cong f_*(L\otimes f^*M)$ for $M\in \Pic^0(A)$. By \cite[\S 15, Theorem 1]{mumfordbook}, the pullback $f^\ast: \Pic^0(A) \to \Pic^0(A')$ has finite kernel. The image of $f^\ast$ is therefore a closed subvariety of $\Pic^0(A')$ of the same dimension, which implies that $f^\ast$ is surjective. 
Hence every simple semi-homogeneous bundle on $A$ of slope $H$ is of the form $f_*(L')$, for a line bundle $L'$ on $A'$ algebraically equivalent to $L$, and, conversely, $f_*(L')$ is simple semi-homogeneous of slope $H$ for all $L'$ algebraically equivalent to $L$.
Moreover, Lemma \ref{lem:slope of push-forward} implies that $L$ has class $f^*H$ in $\NS(A')_\Q$. Because $\NS(A')$ is torsion free by \cite[\S 19, Corollary 2]{mumfordbook}, the pullback $f^*H$ in fact determines the class of $L$ in $\NS(A')$.  

\subsection{Moduli spaces of semi-homogeneous bundles}
\label{subsec:moduli spaces of semi-homogeneous bundles}

As before, let $A$ be an abelian variety over an algebraically closed field 
$K$ of characteristic $0$ and let $H\in \NS(A)_\Q$. Let $k\in \Z_{>0}$ and set $r=k\cdot n(H)$. In what follows, we show that there exists a connected component $S_{H,k}(A)$ in the moduli space of semistable sheaves on $A$ whose points correspond to $S$-equivalence classes of semi-homogeneous bundles of slope $H$ and rank $r$. Denote by $M_{H,k}(A)$ its normalization. In the case $k=1$, we explicitly construct $M_{H,1}(A)\simeq S_{H,k}(A)$ as a torsor  over an abelian variety. For $k>1$, we construct an isomorphism between the symmetric power $\Sym^k(M_{H,1}(A))$ and $M_{H,k}(A)$. 
We begin with some notation.

\begin{nota}\label{nota2}
    For $H\in \NS(A)$ we will denote by $\Pic^H(A)$ the fiber over $H$ along the quotient map $\Pic(A)\to \NS(A)$. Let $\pi^\ast: A^\vee  \rightarrow A^\vee / \Sigma(H)$ be the isogeny with kernel $\Sigma(H)$, and let $\pi: A_H \rightarrow A$ be the dual isogeny, where $A_H$ denotes the dual of $A^\vee / \Sigma(H)$.
\end{nota}

\begin{proposition}
\label{prop:openness of semi-homogenity}
 Let $\mc F$ be an $S$-flat coherent sheaf on $A\times_K S$ for some finite-type $K$-scheme $S$. Then the set of closed points $s\in S$ such that $\mc F_s$ is semi-homogeneous of slope $H$ and rank $k\cdot n(H)$ is the set of closed points of an open subset of $S$. 
\end{proposition}

\begin{proof}
    Being a locally free sheaf is an open condition by \cite[Lemma 2.1.8]{HuybrechtsLehn}, as is the condition of having slope $H$ and rank $k\cdot n(H)$ 
    We may thus assume that $\mc F$ is a family of vector bundles of slope $H$ and rank $k\cdot n(H)$. 

    By \cite[Proposition 5.4]{Mukai}, a vector bundle $E$ on $A$ is semi-homogeneous if and only if $\pi^*E$ is a semi-homogeneous vector bundle on $A_H$. If $E$ is of slope $H$, then $\pi^*E$ is of slope $\pi^*H$ and by \cite[Proposition 7.6]{Mukai} we have $\pi^*H\in \NS(A)$. Replacing $A$ by $A_H$ we may thus assume, from the onset, that $H\in \NS(A)$.

    Let $\mc P_H$ denote the universal line bundle on $A\times \Pic^H(A)$ and let
    \begin{equation*}
    \begin{split}
    \Phi \colon D^b\big(\Pic^H(A)\big)&\longrightarrow D^b(A) \\ \mc K^\bullet &\longmapsto  \mathbf R p_{1*} (p_{2}^*\mc K^\bullet\otimes^{\mathbf L} \mc P_H)
    \end{split}
    \end{equation*}
    denote the associated Fourier-Mukai transform. By Theorem \ref{thm:decomposition of semi-homogeneous vector bundles}, semi-homogeneous bundles on $A$ of slope $H$ are precisely the successive extensions of elements of $\Pic^H(A)$. These correspond under $\Phi$ to the successive extension of skyscraper sheaves on $A$, that is to zero-dimensional coherent sheaves on $A$. Applying the inverse $\Phi^{-1}$ to $\mc F$ (this can be done in families, see \cite[Section 1.2.1]{FourierMukaiBook}), it suffices to prove the following: given an element $\mc K^\bullet\in D^b(\Pic^H(A)\times_K S)$, the set $U$ of $s\in S$ such that $\mathbf L i_s^*\mc K^\bullet$ is finite of length $k$ is open. Here, we denote by $i_s\colon \Pic^H(A)\times_K \kappa(s) \to \Pic^H(A)\times_K S$ the inclusion. To prove openness, it suffices to show that $U$ is both constructible and stable under generalization. 

    We first prove that $U$ is stable under generalization. Suppose that $s\in S$ specializes to $t\in U$, let $R$ be the local ring of $t$ in $\overline{\{s\}}$, and let $\iota\colon \Spec R\to S$ denote the inclusion. Then after replacing $S$ by $\Spec R$ and $\mc K^\bullet$ by $\mathbf L \iota^*\mc K^\bullet$, we may assume that $t$ is the unique closed point of $S$ and that $\mathbf L i_t^*\mc K^\bullet$ is of finite length $k$. By \cite[Lemma 3.31]{Huybrechts}, it follow that $\mc K^\bullet$ is isomorphic to a sheaf $\mc L$ that is flat over $S$. Since Hilbert polynomials are locally constant in flat families, it follows that $i_s^*\mc L\cong \mathbf Li_s^*\mc K^\bullet$ is also of finite length $k$. 

    To prove constructibility, by Noetherian induction it suffices  to prove that $U\cap V$ is open for some nonempty open subset $V\subset S$. By generic flatness, we may choose $V$ to be a non-empty open subset such that $\mc H^i(\mc K^\bullet)\vert_V$ is flat over $V$ for all $i\in \Z$. Then for all $s\in S$ we have $\mc H^i(\mathbf Li_s^*\mc K^\bullet)= i_s^*\mc H^i(\mc K^\bullet)$, and therefore the set $W$ consisting of all $s\in V$ such that $\mathbf Li_s^*\mc K^\bullet$ is a sheaf, is open. As $\mc K^\bullet$ is a flat sheaf over $W$ by construction, the set of all $s\in W$ such that $\mathbf L i_s^*\mc K^\bullet$ is a finite sheaf of length $k$ is open.
\end{proof}

\begin{definition}
    We denote by $S_{H,k}(A)$ the open subset of the moduli space of semistable sheaves whose closed points correspond to semi-homogeneous vector bundles of slope $H$ and rank $k\cdot n(H)$ and by  $M_{H,k}(A)$ its normalization.
\end{definition}

By \cite[Proposition 6.13]{Mukai}, semi-homogeneous vector bundles are semistable, and by Theorem \ref{thm:description of moduli spaces} below the points $M_{H,k}(A)$ are in fact in bijection with all $S$-equivalence classes of semi-homogeneous vector bundles of slope $H$ and rank $k\cdot n(H)$. 

We now give an explicit description of $M_{H,1}(A)$ and, subsequently, of $M_{H,k}(A)$. Let $f\colon A'\to A$ be an isogeny such that there exists a line bundle $L$ on $A'$ for which $E=f_*L$ is simple of slope $H$. Following Notation \ref{nota2}, let  $\pi^\ast: A^\vee  \rightarrow A^\vee / \Sigma(H)$ be the isogeny with kernel $\Sigma(H)$, and let $\pi: A_H \rightarrow A$ be the dual isogeny, where $A_H$ denotes the dual of $A^\vee / \Sigma(H)$.  If $M\in \ker(f^*)\subseteq \Pic^0(A)= (A)^\vee(K)$,
then $E\otimes M\cong E$ and hence $M\in \Sigma(H) = \ker(\pi^\ast)$. This implies that $\pi^*$ factors uniquely through $f^*$. Dually, $\pi$ factors uniquely through $f$, say $\pi=f\circ g$. 

 \begin{lemma}
 \label{lem:quotient of picard by Galois group as simple vector bundles} Recall that $\calS_{H}$ denotes the set of all isomorphism classes of simple semi-homogeneous vector bundles $E$ on $A$  with $\delta(E)=H$.
     The map
     \[
     \Pic^{f^*H}(A')(K)\longrightarrow \calS_H
     \]
     given by $L\mapsto f_*(L)$ 
    induces a bijection
     \[
     \Pic^{f^*H}(A')(K)\big/\Aut(A'/A)\xlongrightarrow{\sim} \calS_H \ .
     \]
    
     \end{lemma} 

 \begin{proof}
     We have already seen that every bundle in $\calS_H$ 
     is isomorphic to $f_*(L)$ for some $L\in \Pic^{f^*H}(A')(K)$. If $f_*(L)\cong f_*(L')$ for two line bundles $L$ and $L'$ in $\Pic^{f^*H}(A')(K)$, then 
    \[
    \bigoplus_{g\in \Aut(A'/A)} g^*L\cong f^*f_*L \cong f^*f_* L'\cong \bigoplus_{g\in \Aut(A'/A)}g^*L' \ ,
    \]
     where the first and last isomorphisms exist because $f$ is an $\Aut(A'/A)$-torsor.
     It follows that $L\cong g^*L'$ for some $g\in \Aut(A'/A)$. 
 \end{proof}

\begin{lemma}
 \label{lem:quotient of picard by Galois group as Picard group of large cover}
 The pull-back $g^*$ induces an isomorphism
 \[
 \Pic^{f^*H}(A')\big/\Aut(A'/A)\cong \Pic^{\pi^*H}(A_H)
 \]   

 \end{lemma}

 \begin{proof}
     First note that since $\Pic^{f^*H}(A')$ is projective and $\Aut(A'/A)$ is finite, the geometric quotient $\Pic^{f^*H}(A')\big/\Aut(A'/A)$ exists by \cite[Exposé V, \S1]{SGA1} or \cite[\href{https://stacks.math.columbia.edu/tag/07S7}{Tag 07S7}]{stacks-project}.
     Let $L\in \Pic^{f^*H}(A')(K)$. Then every other $K$-point of $\Pic^{f^*H}(A')$ is represented by $L\otimes f^*M$ for some $M\in \Pic^0(A)$. 
     Suppose $g^*(L)\cong g^*(L\otimes f^*M)$. Then $\pi^*M\cong \mc O_{A_H}$ and hence $M\in \Sigma(H)$. It follows that $f_*(L)\cong f_*(L\otimes f^*M)$. Using Lemma \ref{lem:quotient of picard by Galois group as simple vector bundles}, we conclude that $g^*$ induces a morphism $\Pic^{f^*H}(A')/\Aut(A'/A)\to \Pic^{\pi^*H}(A_H)$ that is bijective on $K$-points. Since $\Pic^{\pi^*H}(A_H)$ is normal, this concludes the proof. 
 \end{proof}

\begin{remark}\label{definitionofpsi}
 By Lemma \ref{lem:quotient of picard by Galois group as simple vector bundles} and by Lemma \ref{lem:quotient of picard by Galois group as Picard group of large cover} we have a bijection
\[
\psi_H\colon \calS_{H}\xlongrightarrow{\cong} \Pic^{\pi^*H}(A_H)(K) \ .
\]
We note that a priori, the bijection $\psi_H$ depends on $f$: one first represents $E\in \calS_{H}$ as $f_*(L)$ for some $L\in \Pic^{f^*H}(A')$, and then defines $\psi_H(E)=g^*L$. But then $f^*E\cong \bigoplus_{g\in \Aut(A'/A)} g^* L$ because $f$ is an $\Aut(A'/A)$-torsor, and therefore $\pi^*E\cong \psi_H(E)^{\rk(E)}$ by the argument above. This condition uniquely determines $\psi_H$. 
\end{remark}


\begin{theorem}
\label{thm:description of moduli spaces}
Keep notations as above. The space $S_{H,k}(A)$ is a connected component in the moduli space of semi-stable torsion-free sheaves.   
\begin{enumerate}[(i)]
\item When $k=1$, then we have natural isomorphisms
\begin{equation*}
M_{H,1}(A) \cong S_{H,1}(A)\cong\Pic^{\pi^*H}(A_H) 
\end{equation*}
\item When $k\geq 1$, then there is a natural bijective morphism
\begin{equation*}
    \phi_k\colon\Sym^k M_{H,1}(A)\longrightarrow S_{H,k}(A)
\end{equation*}
that induces an isomorphism
\begin{equation*}
    \Sym^k M_{H,1}(A)\xlongrightarrow{\sim} M_{H,k}(A) \ .
\end{equation*}
\end{enumerate}

\end{theorem}

\begin{proof}
Let $\mc P_H$ denote the universal line bundle on $A'\times_K\Pic^{f^*H}(A')$ (whose fiber over a point $[L]\in \Pic^{f^*H}(A')(K)$ is $L$). 
Define
\[
\mc E_1= (f\times \id)_*\mc P_H \ .
\]
This is a vector bundle on
$ A\times_K\Pic^{f^*H}(A')$ whose fiber over a point $[L]\in \Pic^{f^*H}(A')(K)$ is the simple semi-homogeneous vector bundle $f_*L$.
Let 
\[
\phi_1'\colon \Pic^{f^*H}(A')\longrightarrow S_{H,1}(A)
\]
be the associated morphism.
For $\lambda\in \Aut(A'/A)$, the fiber of $\mc E_1$ over $T_\lambda^*[L]=[T_\lambda^*L]$ is isomorphic to the fiber over $[L]$, because we have
\[
f_*(T_\lambda^*L)\cong f_*(T_{-\lambda})_*L=f_*L \ .
\]
It follows that $(\id\times T_{\lambda}^*)^* \mc E_1$ and $\mc E_1$ have isomorphic fibers over all points of $\Pic^{f^*H}(A')$. As $\Pic^{f^*H}(A')$ is reduced, we can conclude from \cite[Theorem 1.8]{Mukai} (see also \cite[Lemma 4.6.3]{HuybrechtsLehn} for the assumption on simplicity over the base) that there exists a line bundle $\mc L_\lambda$ on $\Pic^{f^*H}(A')$ such that 
\[
(\id\times T_{\lambda}^*)^* \mc E_1\cong \mc E_1\otimes p_2^*\mc L_\lambda \ , 
\]
where $p_2$ denotes the projection to $\Pic^{f^*H}(A')$. 
This shows that $\phi_1'$ is $\Aut(A'/A)$-equivariant and thus descends, by Lemma \ref{lem:quotient of picard by Galois group as Picard group of large cover}, to a morphism 
\[
    \phi_1\colon \Pic^{\pi^*H}(A_H) \to S_{H,1}(A) \ . 
\]
For $F\in \mc S_H$, the morphism $\phi_1$ maps the point $\psi_H(F)\in \Pic^{\pi^*H}(A_H)(K)$ to the point of $\Pic^{\pi^*H}(A_H)$ corresponding to $F$. We conclude that $\phi_1$ is bijective on $K$-points.
Moreover, the space $S_{H,1}(A)$ is the image of the proper space $\Pic^{\pi^*H}(A_H)$ and is therefore closed in the moduli space of semistable sheaves, as well as connected. By definition, it is also open, showing that it is a component. By \cite[Proposition 3.1]{Gulbrandsen}, the component $S_{H,1}(A)$ is smooth and, thus, automatically isomorphic to $M_{H,1}(A)$. The bijectivity of $\phi_1$ on $K$-points implies that $\phi_1$ is an isomorphism by Zariski's main theorem.

For the case $k>1$, consider the vector bundle
\[
\mc E_k\coloneqq \bigoplus_{1=1}^k p_i^*\mc E_1
\]
on $(\Pic^{f^*H}(A))^k\times_K A$. The $\Aut(A'/A)$-action on $\Pic^{f^*H}(A)$ and the $S_k$-action on $(\Pic^{f^*H}(A))^k$ combine to an $(S_k\ltimes \Aut(A'/A)^k)$-action and it follows from the $k=1$ case that the morphism
\[
\phi_k'\colon (\Pic^{f^*H}(A))^k\longrightarrow S_{H,k}(A)
\]
induced by $\mc E_k$ descends to a morphism
\[
\phi_k \colon \Sym^k(M_{H,1}(A))=(\Pic^{f^*H}(A))^k /(S_k\ltimes \Aut(A'/A)^k) \longrightarrow S_{H,k}(A) \ .
\]
By Theorem \ref{thm:decomposition of semi-homogeneous vector bundles} and the $k=1$ case, the morphism $\phi_k$ is bijective on $K$-points. 
The statement about being a component follows as in the $k=1$ case. That $\phi_k$ induces an isomorphism between $\Sym^k(M_{H,1}(A))$ and the normalization $M_{H,k}(A)$ follows from Zariki's main theorem.
\end{proof}

\begin{remark}
    Suppose that $K$ is complete with respect to a non-Archimedean absolute value. Because $\phi_k$ is bijective, the same is true for the induced morphism $\phi_k^\an\colon\Sym^k M_{H,1}(A)^{\an}\rightarrow S_{H,k}(A)^{\an}$ of Berkovich analytic spaces by \cite[Proposition 3.4.6]{Berkbook}. Since $\Sym^k(M_{H,1}(A))^\an$ is compact by \cite[Theorem 3.4.8]{Berkbook} and $M_{H,k}(A)^\an$ is separated by \cite[Proposition 3.1.5]{Berkbook}, it follows that $\phi_k^\an$ is a homeomorphism. Hence our result about the skeleton of $M_{H,k}(A)^{\an}$ in Theorem \ref{mainthm_skeleton=tropicalization} is in fact also a statement about $S_{H,k}(A)^{\an}$.
\end{remark}

\section{Tropical and non-Archimedean Appell--Humbert}\label{section_AppellHumbert}


In this section,  let $K$ be an algebraically closed field of characteristic $0$, complete with respect to a non-trivial valuation $\nu\colon K\rightarrow \R\cup\{\infty\}$ and let $A$ be an abelian variety over $K$. 
For a scheme $Y$ locally of finite type over $K$ we denote by  $Y^{\an} $ the associated Berkovich analytic space.
We assume that $A$ has totally degenerate reduction, i.e.\ that there exists a uniformization $A^\an \cong T^\an/\Lambda$ for some torus $T$ over $K$
and some (algebraic) lattice $\Lambda$ in $T(K)$. Let $M$ denote the character lattice of $T$ and let $N=\Hom(M,\Z)$. We have a tropicalization map $T^\an\to N_\R$ given by mapping $x\in T^{\an}$ to the homomorphism $m\mapsto-\log\vert \chi^m\vert_x$ in $N_\R=\Hom_{\Z}(M,\R)$. The lattice $\Lambda \subseteq T(K)$ maps isomorphically to a lattice in  $N_\R$, which by abuse of notation we also denote by $\Lambda$. 
The tropicalization of $A$ is the tropical abelian variety $A^\trop= N_\R/\Lambda$. The tropicalization map on $T$ induces a natural continuous tropicalization map $A^{\an}=T^{\an}/\Lambda\rightarrow A^{\trop}=N_\R/\Lambda$ (see \cite{GublerTropicalAbelian} for details).

By \cite[Proposition 3.4.11]{Berkbook}, we have $\Pic(A)=\Pic(A^\an)$, which allows us to express $\Pic(A)$ explicitly in terms of the uniformization of $A^\an$ via the non-Archimedean Appell-Humbert theorem  \cite{Gerritzen}. Similarly, tropical line bundles on $A^\trop$ can be described explicitly via the tropical Appell-Hubert theorem \cite[Theorem C]{Gross_Shokrieh_Tautological}. We recall these principles in detail below, starting with the tropical side.

Let us recall the notion of tropical line bundles on the integral affine manifold $A^\trop$.

\subsection{Integral affine manifolds}

An \emph{integral affine structure} on a topological manifold $X$ is a sheaf $\Aff_X$ of continuous functions on $X$ such that each $x\in X$ has an open neighborhood $U$ that can be embedded openly into $\R^n$ for some $n\in \Z_{\geq 0}$ in a way that identifies the sections of $\Aff_X$ with functions on (open subsets of) $U$ that are locally of the form $x\mapsto \langle m,x\rangle+ s$ for some $m\in (\Z^n)^\vee$ and $s\in \R$. Here $\langle -,-\rangle \colon (\Z^n)^\vee\times \Z^n\to \Z$ denotes the duality pairing that we extend $\R$-linearly to a pairing $(\R^n)^\vee\times \R^n\to \R$. A topological manifold together with an integral affine structure is called an \emph{integral affine manifold}. A \emph{morphism} between two integral affine manifolds $X$ and $Y$ is a continuous map $f\colon X\to Y$ that induces, via pull-back, a morphism $f^{-1}\Aff_Y\to \Aff_X$. 

If $Y$ is a topological manifold and $f\colon X\to Y$ is a covering space, then an integral affine structure $\Aff_Y$ on $Y$ induces an integral affine structure on $X$ by defining $\Aff_X$ as the sheaf of continuous functions generated by all functions of the form $\phi\circ f\vert_{f^{-1}U}$, where $\phi\in \Gamma(U,\Aff_Y)$ for some open subset $U\subseteq Y$. Conversely, if $(X,\Aff_X)$ is an integral affine manifold such that the deck transformations are morphisms of integral affine manifolds, then there is an induced integral affine structure on $Y$ obtained by defining $\Gamma(U,\Aff_Y)$ for $U\subseteq Y$ open as the set of precisely those continuous functions $\phi\colon U\to \R$ for which $\phi\circ f\vert_{f^{-1}U}\in \Gamma(f^{-1}U,\Aff_X)$. 

\subsection{Tropical line bundles}
A \emph{tropical line bundle} on an integral affine manifold $X$ is an $\Aff_{X}$-torsor. We can identify the set of isomorphism classes of $\Aff_{X}$-torsors with  $H^1(X, \Aff_X)$ in the usual way, allowing us to define the \emph{Picard group} $\Pic(X)$ of $X$ as
\[
\Pic(X) \coloneqq H^1(X, \Aff_{X}) \ .
\]

Let $\R_{X}$ denote the sheaf of locally constant real valued functions on $X$. These functions are affine and we define the tropical cotangent bundle on $X$ via the short exact sequence 
\[
0\longrightarrow \R_{X}\longrightarrow \Aff_{X}\longrightarrow \Omega^1_{X}\longrightarrow 0 \ ,
\]
called the \emph{tropical exponential sequence}. The \emph{first Chern class} $c_1(L)$ of an element $L\in \Pic(X)$ on $X$, as introduced in \cite{MikhalkinZharkov, JellRauShaw}, is the image of $L$ under the morphism
\[
c_1\colon H^1(X, \Aff_{X})\longrightarrow H^1(X,\Omega^1_{X}) 
\]
that is induced by the quotient map. We denote the kernel of this morphism by $\Pic^0{X}$.

\subsection{Factors of automorphy on real tori}

Now let $X=A^\trop$. Any isomorphism $N\xrightarrow{\cong}\Z^n$ makes $N_\R$ an integral affine manifold. This integral affine structure does not depend on the chosen isomorphism. The deck transformations of the covering $N_\R\to A^\trop= N_\R/\Lambda$ are the translations by elements in $\Lambda$, which are automorphisms of the integral affine manifold $N_\R$. Therefore, there is an induced integral affine structure $\Aff_{A^\trop}$ on $A^\trop$. 
For $X=A^\trop$, the sheaf $\Omega^1_{A^\trop}$ is isomorphic to the constant sheaf with values in the character lattice $M$, so we can identify
\[
H^1(A^\trop,\Omega^1_{A^\trop})\cong  H^1(N_\R/\Lambda ,M) \cong \Lambda^\vee \otimes_\Z M \ ,
\]
which can in turn be identified with the group of homomorphisms $H\colon \Lambda\to M$. For $\lambda,\lambda'\in \Lambda$ we denote  by $[\lambda,\lambda']_H^\R$ the real number $\big\langle \lambda, H(\lambda')\big\rangle$, where $\langle -, -\rangle$ denotes the evaluation pairing $N_\R\times M_\R\to \R$. The image of $c_1$ consists precisely of those $H$ whose associated bilinear form $[-,-]^\R_H$ on $\Lambda$ is symmetric (see \cite[p.\ 15]{MikhalkinZharkov} or \cite[Theorem C and  Proposition 7.1]{Gross_Shokrieh_Tautological}). In this case, we say that $H$ is \emph{$\R$-symmetric}.

\begin{definition} Let $A^{\trop}$ be a real torus with integral structure. 
\begin{enumerate}[(i)]
\item The \emph{tropical Néron-Severi group}, denoted $\NS(A^\trop)$ is the group of $\R$-symmetric homomorphisms $\Lambda\to M$. For every $H\in \NS(A^\trop)$ we denote $\Pic^H(A^\trop)=c_1^{-1}\{H\}$.

\item A \emph{factor of automorphy on $A^\trop$} is a pair  $(H,l)$ consisting of an $\R$-symmetric bilinear homomorphism $H\colon \Lambda\to M$ and a morphism $l\in \Hom(\Lambda,\R)$.
\end{enumerate}
\end{definition}

 By \cite[Theorem C]{Gross_Shokrieh_Tautological}, every factor of automorphy $(H,l)$ determines a line bundle on $A^\trop$ which we denote by $L(H,l)$, and every line bundle on $A^\trop$ is of this form. We can recover $H$ from $L(H,l)$ via the identity $c_1\big(L(H,l)\big)= H$. In particular, the class $H$ is uniquely determined by the line bundle $L(H,l)$. The element $l$, on the other hand, is not uniquely determined. But we have $L(H,l)\cong L(H,l')$ if and only if $l_\R-l_\R'$ has integer values on $N$, that is $l_\R-l_\R'\in M$, where the subscript $\R$ indicates that we extend $l$ and $l'$ linearly over $\R$ to $N_\R$. We thus obtain a natural identification
\[
\Pic^0(A^\trop)= \Hom(N_\R,\R)/M = \Lambda^\vee_\R/N^\vee \ .
\]
with the \emph{dual torus} $\Lambda^\vee_\R/N^\vee$.
Here, we identify $\Lambda^\vee_\R$ first with $\Hom(\Lambda,\R)$ and then with $\Hom_\R(N_\R,\R)$.

For each $x\in A^\trop$ denote the translation by $x$ as $T_x\colon A^\trop\to A^\trop,\;\; y\mapsto x+y$.
Translations of line bundles can be described explicitly in terms of factors of automorphy. Namely, if $x \in N_\R$ represents $\overline x\in A^\trop$, then $T_{\overline x}^{-1}L(H,l) =L\big(H, l-H(x)\big)$ by \cite[Proposition 7.5]{Gross_Shokrieh_Tautological}.

\subsection{Line bundles on $A^\an$}

As in the tropical case,  every line bundle on the abelian variety $A^\an \simeq T^\an / \Lambda$ is uniquely determined by a \emph{factor of automorphy}, i.e a pair $(H, r)$ consisting of a homomorphism
\[
H\colon \Lambda  \longrightarrow M
\]
and a map $r\colon \Lambda \to \G_{m,K}$ such that
\begin{equation}
\label{eq:condition on non-Arch factor of automorphy}
\big\langle \lambda , H(\lambda')\big\rangle =\frac{r(\lambda+\lambda')}{r(\lambda)\cdot r(\lambda')} \ ,    
\end{equation}
where $\langle -,-\rangle$ denotes the evaluation pairing of a character in $M$ on an element of $T$. 
We denote $\big\langle \lambda , H(\lambda')\big\rangle$ by $[\lambda,\lambda']_H$. 
We say that a morphism $H\colon \Lambda\to M$ is \emph{$\G_m$-symmetric} if 
\[
[\lambda,\lambda']_H=[\lambda',\lambda]_H
\]
for all $\lambda,\lambda'\in \Lambda$. Note that if $H$ is part of a factor of automorphy, then it is $\G_m$-symmetric because it satisfies \eqref{eq:condition on non-Arch factor of automorphy}. Conversely, every $\G_m$-symmetric bilinear map $H\colon \Lambda \to M$ appears as a factor of automorphy by \cite[Lemma 2.3]{bolu}.  
Because $\nu[\lambda,\lambda']_H = [\lambda,\lambda']_H^\R$, every $\G_m$-symmetric morphism $H\colon \Lambda\to M$ is automatically $\R$-symmetric.

Let $L_{(H,r)}$ denote the line bundle on $A^\an$ corresponding to the factor of automorphy $(H,r)$. We have 
\[
L_{(H,r)} \cong L_{(H',r')}
\]
if and only if $H=H'$ and there exists $m\in M$ such that 
\[
\frac{r'(\lambda)}{r(\lambda)} = \langle \lambda, m \rangle
\]
for all $\lambda\in\Lambda$. 
Moreover, if $x\in T(K)$ and $\overline x$ is its image in $A(K)$, then 
\[
T_{\overline x}^* L_{(H, r)} \cong L_{(H, r')}
\]
where 
\[
\frac{r'(\lambda)}{r(\lambda)} =\big\langle x,H(\lambda)\big\rangle \ .
\]

The elements of $\Pic^0(A)=\Pic^0(A^\an)$ are precisely the translation invariant line bundles, which are precisely those that can be written as $L_{(0,r)}$ for some group homomorphism $r\colon  \Lambda \to \G_m(K)$. It follows that the Néron-Severi group $\NS(A)$ is isomorphic to the group of $\G_m$-symmetric morphisms $\Lambda\to M$.

As noted above, $\G_m$-symmetric morphisms $\Lambda\to M$ are $\R$-symmetric, so there is an inclusion $\NS(A)\to \NS(A^\trop)$. However, as the next example shows that $\NS(A)$ is \emph{not} saturated in $\NS(A^\trop)$. 

\begin{example}
\label{ex:algebraic NS not saturated}
Let $\Lambda_a$ be the lattice spanned by 
\[
\lambda_1=\left(\begin{matrix}
t\\1
\end{matrix} \right),
~ \lambda_2=\left(\begin{matrix}
-1 \\t 
\end{matrix}\right)
\]
in $\G_{m}^2(K)$, where $K$ is a complete algebraically closed extension field of $\C((t))$, and consider the analytic torus $B=\G_m^{2,\an}/\Lambda_a$.
Let $e_1, e_2$ be the standard basis of $N=\Z^2$ with dual basis $e_1^*,e_2^*$, and let $H\colon \Lambda\to M$ be given by $H(\lambda_i)=e_i^*$.
Clearly, $H\in \NS(B^\trop)$, but we have
\begin{equation}
\label{eq:no maximal sublattice}\begin{split}
[\lambda_1, \lambda_2]_H &=1 \qquad \text{as well as} \\
[\lambda_2,\lambda_1]_H &=-1 \ .
\end{split}
\end{equation}
So $H$ is not $\G_m$-symmetric. But $2H$ is $\G_m$-symmetric as $[\lambda_2,\lambda_1]_{2H}=(-1)^2=1$. We conclude that $H\in \NS(B^\trop)\setminus \NS(B)$ with $2H\in \NS(B)$. Note that since $[-,-]_{2H}^\R$ is positive definite, the  analytic torus $B$ is in fact algebraic, that is there exists an abelian variety $A$ over $K$ with $A^\an=B$ (see \cite[Theorem 6.4.4]{Luetkebohmert}). 
\end{example}

In what follows, we identify the elements of $\NS(A)_\Q$ with morphisms $\Lambda\to M_\Q$. 

\begin{definition}\label{definition: Large lattice {H}}
    Let $H\in \NS(A)_\Q$. We define the subgroup $\LargeLattice{H}$ of $\Lambda$ by
    \[
        \LargeLattice{H}=\big\{ \lambda\in \Lambda \ \big\vert \  H(\lambda)\in M \big\}. 
    \]
\end{definition}

\begin{lemma}
\label{lem:maximal lattice on tropical side}
Let $H\in \NS(A)_{\Q}$. Then $\LargeLattice{H}$ has finite index in $\Lambda$.
\end{lemma}

\begin{proof}
The subgroup $\LargeLattice{H}$ of $\Lambda$ is precisely the kernel of the composite 
    \[
    \Lambda\xlongrightarrow{H} M_\Q\longrightarrow M_\Q/M
    \]
    Since every element in $M_\Q/M$ has finite order, the quotient $\Lambda/\LargeLattice{H}$ is a finitely generated abelian group, all of whose elements have finite order. Therefore, the quotient $\Lambda/\LargeLattice{H}$ is finite.
\end{proof}

While  $H\in \NS(A)_\Q$ does not in general induce a pairing $\Lambda\times\Lambda\to \G_m$ (unless $H\in \NS(A)$), it does make sense to define a bilinear pairing
\begin{equation*}\begin{split}
[-,-]_H\colon \Lambda\times \LargeLattice{H}&\longrightarrow \G_m \\ (\lambda,\lambda')&\longmapsto \big\langle \lambda, H(\lambda')\big\rangle 
\end{split}\end{equation*}
that agrees with the pairing defined above if $H\in \NS(A)$.

This leads us to defining the following subgroup of $\LargeLattice{H}$, which will play a key role in section \ref{section_uniform+semihom}.  

 \begin{definition}\label{smalllattice{H}}
    We define the subgroup $\SmallLattice{H}$ of $\LargeLattice{H}$ by
    \[
    \SmallLattice{H}=\big\{\gamma\in \LargeLattice{H}\ \big\vert\ [\gamma ,\lambda]_H=[\lambda,\gamma]_H \text{ for all }\lambda\in \LargeLattice{H}\big\}.
    \]
 \end{definition}
 
 \begin{lemma}
     \label{lem:Gamma contains lLambda}
     Let $l$ be a positive integer such that $lH\in \NS(A)$. Then $l\Lambda\subseteq \SmallLattice{H}$. In particular, the sublattice $\SmallLattice{H}$ has finite index in $\Lambda$.
 \end{lemma}
 
 \begin{proof}
    Let $\gamma=l\theta\in l\Lambda$ and let $\lambda\in \LargeLattice{H}$. Then we have
    \[
        [\gamma,\lambda]_H= [\theta,\lambda]_{lH}=[\lambda,\theta]_{lH}=[\lambda,\gamma]_H \ .
    \]
    It immediately follows that $l\Lambda\subseteq \SmallLattice{H}$.
 \end{proof}

\section{Semi-homogeneous tropical bundles on real tori}
\label{section_tropicalvectorbundles}


Let $\Sigma=N_\R/\Lambda$ be a real torus with integral structure and write $M=\Hom(N,\Z)$. In the previous section, we have studied line bundles on $\Sigma$. In this section we introduce and study vector bundles of higher rank on $\Sigma$. 

\subsection{Tropical vector bundles on integral affine manifolds}
Let $X$ be an integral affine manifold, whose sheaf of affine functions is denoted by $\Aff_X$. 
Denote by $\T=\R\cup\{\infty\}$ the semifield of tropical numbers (with $\min$ and $+$ as operations). It is well-known that the group of invertible tropical $r\times r$-matrices precisely consists of all $r\times r$-matrices with exactly one entry in all rows and columns not equal to $\infty$ (see \cite[Lemma 1.4]{Allermann}). In other words, we have $\GL_r(\T)=S_r\ltimes\R^r$. Our approach, here and in \cite{GrossUlirschZakharov}, is to think of a vector bundle on $X$ as a principal $\GL_r(\T)$-bundle in terms of its affine transition maps with integral slopes taking values in $\GL_r(\T)=S_r\ltimes\R^r$.

\begin{definition}
    Let $(X,\Aff_X)$ be an integral affine manifold. A \emph{tropical vector bundle} of rank $r$ on $X$ is an $S_r\ltimes \Aff_X^r$-torsor on $X$.  
\end{definition}

A \emph{morphism} of tropical vector bundles is a morphism of $S_r\ltimes \Aff_X^r$-torsors. In particular, all morphisms of tropical vector bundles are isomorphisms.

\begin{definition}
    A \emph{free cover} of an integral affine manifold $X$ is a morphism $\pi\colon Y\to X$ of integral affine manifolds which is a covering space 
    for the underlying topological spaces and induces an isomorphism
    \[
    \pi^{-1}\Aff_X\longrightarrow \Aff_Y \ .
    \]
\end{definition}

\begin{example}
    The morphism $\pi\colon \R/\Z\rightarrow \R/\Z$ of real tori with integral structure given by
    \[
    x+\Z\longmapsto 2x+\Z
    \]
     is a covering space on the underlying topological spaces but not a free cover: the pull-backs of affine functions via $\pi$ do not generate $\Aff_{\R/\Z}$ because their slopes are always even.
\end{example}

\begin{proposition}\label{prop_vectorbundle=cover+linebundle}
The category of tropical vector bundles of rank $r\in \Z_{\geq 1}$ on an integral affine manifold $X$  is equivalent to the category  of pairs $(Y\xrightarrow{\pi} X,L)$ consisting of a free cover $\pi$ of degree $r$ and a tropical line bundle $L$ on $Y$.
\end{proposition}

\begin{proof}
Let $\mc G$ denote the category fibered in groupoids over the small site $O(X)$ of open subsets of $X$ consisting of pairs $(Y\xrightarrow{\pi}U ,L)$, where $\pi$ is a free cover of degree $r$ of an open subset $U$ of $X$ and $L$ be a tropical line bundle on $Y$.  Note that this is equivalent to the category  consisting of pairs $(Y\xrightarrow{\pi}U,L)$, where $\pi$ is a cover of the underlying topological space $U$ and $L$ is a $\pi^{-1}\Aff_U$-torsor. Since covering spaces and torsors glue, the fibered category $\mc G$ is a stack over $O(X)$. Let $I\in \mc G(X)$ denote the trivial degree-$r$ cover of $X$ equipped with the trivial torsor. Then every element of $\mc G$ is locally isomorphic to $I$. In other words, the stack $\mc G$ is a neutral gerbe on $O(X)$. It follows that there is an equivalence of gerbes
\[
\mc G \xlongrightarrow{\sim} \mathrm{Tors}\big(\underline{\mathrm{Isom}}(I,I)\big)\qquad \textrm{given by}\qquad J\longmapsto \underline{\mathrm{Isom}}(I,J) \ ,
\]
 where $\mathrm{Tors}\big(\underline{\mathrm{Isom}}(I,I)\big)$ denotes the gerbe of $\underline{\mathrm{Isom}}(I,I)$-torsors. In particular, we obtain an equivalence of categories between $\mc G(X)$ and the category of torsors over $\underline{\mathrm{Isom}}(I,I)= S_r\ltimes \Aff_X^r$, which are precisely the tropical vector bundles of rank $r$ on $X$.
\end{proof}

\begin{remark}
Isomorphism classes of tropical vector bundles can be described via transition maps that satisfy the cocycle condition. Given transition maps for a tropical vector bundle $E$, one can explicitly construct a cover $Y\xrightarrow{\pi} X$ and a tropical line bundle $L$ on $Y$ such that $f_*L$ is isomorphic to $X$. This has been carried out on tropical curves in the proof of \cite[Proposition 3.2]{GrossUlirschZakharov}.
\end{remark}

\begin{definition}
    Let $X$ be an integral affine manifold. If $f\colon Y\to X$ is a free cover and $L$ is a tropical line bundle on $Y$, we denote by $E(f,L)$ the tropical vector bundle corresponding to $f$ and $L$ (as in Proposition \ref{prop_vectorbundle=cover+linebundle}). 
    We define direct sums, tensor products, pull-backs, and push-forwards of tropical vector bundles as follows:
    
    \begin{itemize}
        \item The \emph{direct sum} of $E_i=E(Y_i\xrightarrow{f_i} X, L_i)$ for $i\in\{1,2\}$ is given by
        \[
        E_1\oplus E_2 = E\big(Y_1\sqcup Y_2 \xrightarrow{f_1\sqcup f_2}X, L_1\sqcup L_2\big) \ .
        \]
        
        \item The \emph{tensor product} of $E_i=E(Y_i\xrightarrow{f_i} X, L_i)$, $i\in \{1,2\}$ is given by
        \begin{equation*}
            E_1\otimes E_2 = E\big(f_1\times_X f_2, g_1^{-1}L_1\otimes g_2^{-1}L_2\big) \ .
        \end{equation*}
        Here $f_1\times_X f_2$ denotes the cover $Y_1\times_X Y_2\rightarrow X$ induced from both $f_1$ and $f_2$ and $g_i$ is the projection from $Y_1\times_X Y_2$ onto $Y_i$.
        
        \item The \emph{pull-back} of $E=E(Y\xrightarrow{f} X,L)$ along a morphism $g\colon Z\to X$ is given by 
        \[
        g^*E=E\big(Y\times_X Z\xrightarrow{f'} Z, (g')^*L\big) \ , 
        \]
        where $f'$ and $g'$ are the projections from $Y\times_X Z$ to $Z$ and $Y$, respectively.
        
        \item The \emph{pushforward} of $E=E(Y\xrightarrow{f}X,L)$ along a free cover $g\colon X\to Z$ of integral affine manifolds is given by
        \[
        f_*E=E(g\circ f, L) \ .
        \]
        Note that this makes sense of the equality $E=f_*L$.
    \end{itemize}
\end{definition}

\begin{definition}
\label{def:indecomposable vector bundle}
    Let $X$ be an integral affine manifold. We say that a tropical vector bundle $E$ is \emph{indecomposable} if there do not exist tropical vector bundles $F_1$ and $F_2$ of positive rank with $E\cong F_1\oplus F_2$. 
\end{definition}

\begin{lemma}
\label{lem:decomposition into indecomposable}
A tropical vector bundle $E(Y\to X, L)$ on an integral affine manifold $X$ is indecomposable if and only if $Y$ is connected. Moreover, every tropical vector bundle $E$ on $X$ can be uniquely expressed (up to permutation) as a sum $\bigoplus_{i=1}^k E_i$ for some indecomposable vector bundles $E_i$
\end{lemma}

\begin{proof}
By definition of sums, the cover associated to a sum $E_1\oplus E_2$ of tropical vector bundles is never connected. Therefore, if $Y$ is connected, then $E(Y\to X,L)$ is indecomposable. To complete the proof, it suffices to show that every tropical vector bundle can be expressed uniquely as a sum of tropical vector bundles corresponding to connected covers. But this is clear because the connected components of a cover are uniquely determined by the cover.
\end{proof}

\subsection{Semi-homogeneous tropical vector bundles}
We now focus on tropical vector bundles on a real torus $\Sigma=N_\R/\Lambda$ with integral structure. Note that the results of this section apply to all real tori with integral structure. We do not need to assume that $N$ and $\Lambda\subseteq N_\R$ are induced from an uniformization of an abelian variety.

\begin{lemma}
\label{lem:covers of tropical Abelian variety}
Let $f\colon \Sigma'\to \Sigma$ be a connected free cover of finite degree. Then there exists a lattice $\Lambda'\subseteq \Lambda$ of finite index such that $f$ is isomorphic to the cover $N_\R/\Lambda'\to N_\R/\Lambda=\Sigma$.
\end{lemma}

\begin{proof}
The projection $N_\R\to \Sigma$ exhibits $N_\R$ as the universal covering space of $\Sigma$ and induces an isomorphism between $\Lambda$ and the fundamental group $\pi_1(\Sigma,0)$ of $\Sigma$. By the classification of covering spaces, the connected finite cover $f$ defines a subgroup $\Lambda'$ of $\pi_1(\Sigma,0)=\Lambda$ of finite index and $\Sigma'$ can be recovered from $\Lambda'$ as the quotient $N_\R/\Lambda'$. 
\end{proof}

Note that as $\Lambda$ is abelian, the cover of $\Sigma$ determined by a sublattice $\Lambda'$ is automatically a $\Lambda/\Lambda'$-torsor.

\begin{remark}\label{identif}
  A finite cover $f\colon \Sigma'=N_\R/\Lambda'\to N_\R/\Lambda=\Sigma$ given by a sublattice $\Lambda'\subseteq \Lambda$ induces a pull-back $f^*\colon \NS(\Sigma)\to \NS(\Sigma')$. The elements of $\NS(\Sigma)$ are $\R$-symmetric homomorphisms $H\colon \Lambda\to M$ and the pull-back of such an $H$ is simply the restriction $H\vert_{\Lambda'}$.  Since $\R$ has no torsion, a homomorphism $H\colon \Lambda\to M$ is $\R$-symmetric if and only if its restriction $H\vert_{\Lambda'}$ is. Now combining this with the fact that $\Lambda'$ has finite index in $\Lambda$ and $\Q$ is divisible, we can conclude that the pull-back induces an isomorphism $\NS(\Sigma)_\Q\xrightarrow{\cong} \NS(\Sigma')_\Q$. 
\end{remark}

\begin{definition}
    Let $E$ be an indecomposable tropical vector bundle on $\Sigma$. Then $E=E(\Sigma'\xrightarrow{f} \Sigma,L)$ for some free cover $f$ with connected domain $\Sigma'$ and some line bundle $L$ on $\Sigma'$. By Lemma \ref{lem:covers of tropical Abelian variety}, we have $\Sigma'=N_\R/\Lambda'$ for some finite-index sublattice $\Lambda'$ of $\Lambda$. As we observed above, the cover $f$ induces an isomorphism
    \[
    \NS(\Sigma)_\Q\xrightarrow{\cong} \NS(\Sigma')_\Q \ .
    \]
    We define the \emph{slope} $\delta(E)$ of $E$ as the unique class in $\NS(\Sigma)_\Q$ pulling back to the class $[L]$ of $L$ in $\NS(\Sigma')_\Q$.
    For a general vector  bundle $E$ on $\Sigma$ with indecomposable summands $E_1,\ldots, E_k$, we define the slope of $E$ as
    \[
    \delta(E)=\frac{\sum_{i=1}^k r(E_i)\delta(E_i)}{r(E)} \ .
    \]
\end{definition}

\begin{remark}
    It is possible to define determinants of tropical vector bundles; the definition given in \cite[Section 2.4]{GrossUlirschZakharov} easily generalizes the case of integral affine manifolds. One can then show that for a tropical vector bundle $E$ on $\Sigma$ one has
    \[
    \delta(E)=\frac{\big[\det(E)\big]}{r(E)} \ .
    \]
\end{remark}

\begin{definition}
For a point $x\in \Sigma$, we denote by $T_x\colon \Sigma\to \Sigma,\; y\mapsto y+x$ the translation by $x$. A tropical  vector bundle $E$ on $\Sigma$ is called \emph{homogeneous} if $T_x^{-1} E \cong E$ for all $x\in \Sigma$. It is called \emph{semi-homogeneous} if for each $x\in \Sigma$ there exists $L\in \Pic^0(\Sigma)$ (possibly depending on $x$) with $T_x^{-1}E \cong E\otimes L$.
\end{definition}

\begin{lemma}
\label{lem:pull-back of translate}
Let $f\colon\Sigma'\to \Sigma$ be a free cover, let $L$ be a line bundle on $\Sigma'$, let $x\in\Sigma$, and let $y\in f^{-1}\{x\}$. Then we have
\[
T_x^{-1} (f_*L)= f_*(T_y^{-1}L)
\]
\end{lemma}

\begin{proof}
The equality 
\[
T_x^{-1}(f_*L) =(T_{-x})_*(f_*L) =f_*\big((T_{-y})_*L\big)= f_*(T_y^{-1} L) \ ,
\]
is a consequence of $T_{-x}\circ f=f\circ T_{-y}$. 
\end{proof}

\begin{lemma}
\label{lem:classification of tropical vector bundles}
Let $f\colon \Sigma'=N_\R/\Lambda'\to N_\R/\Lambda=\Sigma$ be the cover associated to a finite-index sublattice $\Lambda'\subseteq \Lambda$, and let $(H,l)$ and $(H',l')$ be two factors of automorphy for line bundles $\Sigma'$. Then we have $f_*L(H,l)\cong f_*L(H',l')$ if and only if $H=H'$ and $l-l'-H(\lambda)\in M$ for some $\lambda\in \Lambda$.
\end{lemma}

\begin{proof}
Tropical vector bundles are equivalent to isomorphism classes of tropical line bundles on finite free covers. Therefore, the pushforward $f_*L(H,l)$ is isomorphic to $f_*L(H',l')$ if and only if there exists an automorphism of $f$ along which $L(H,l)$ pulls back to $L(H',l')$. The automorphisms of $f$ are precisely the translations by the classes of elements $\lambda\in \Lambda$ and the pull back of $L(H,l)$ via a translation by $\lambda$ is given by $L\big(H, l- H(\lambda)\big)$. This is isomorphic to $L(H',l')$ if and only if $H=H'$ and $l-l'-H(\lambda)\in M$. 
\end{proof}

\begin{proposition}\label{prop_charsemihom}
Let $E_1,\ldots, E_k$ be indecomposable tropical vector bundles on $\Sigma$. Then $E=\oplus_{i=1}^k E_i$ is homogeneous if and only if $\delta(E_i)=0$ for all $1\leq i \leq k$. The vector bundle $E$ is semi-homogeneous if and only if $\delta(E_i)=\delta(E_j)$ for all $1\leq i,j\leq k$.
\end{proposition}

\begin{proof}
For every $i\in\{1,\ldots, k\}$ there exists a finite-index sublattice $\Lambda_i$ of $\Lambda$, a symmetric morphism $H_i\colon \Lambda_i\to M$, and a linear map $l_i\colon \Lambda_i\to \R$ such that 
$E_i =(\pi_i)_*L(H_i, l_i)$, where $\pi_i\colon N_\R/\Lambda_i\to N_\R/\Lambda$ is the quotient map. For $x\in N_\R$, we have 
\[
T_{\overline x}^*E \cong \bigoplus_{i=1}^k  L\big(H_i, l_i-H_i(x)\big)
\]
by Lemma \ref{lem:pull-back of translate}. By Lemma \ref{lem:classification of tropical vector bundles}, this is isomorphic to $E\otimes L(0, l)$ for some $l\colon \Lambda\to \R$ if and only if there exists a permutation $\sigma$ on $\{1,\ldots, k\}$ with $\Lambda_i=\Lambda_{\sigma(i)}$ and $H_i=H_{\sigma(i)}$ as well as for each $1\leq i\leq k$ elements $\lambda_i\in \Lambda$ and $m_i\in M$ with 
\begin{equation}
\label{eq:condition for translate equal to tensor with line bundle}
l_i-H_i(x)= l_{\sigma(i)}+H_i(\lambda_i)+m_i+l    \ .
\end{equation}

If $E$ is homogeneous, we can take $l=0$ and see that there are only countably many choices for $l_i-H_i(x)$ as $i$ varies over $\{1,\ldots, k\}$ and $x$ over $N_\R$. As every one-dimensional subspace of $N_\R$ is uncountable, this implies that $H_i(x)=0$ for all $x\in N_\R$ and $i\in \{1,\ldots, k\}$, and hence $\delta(E_i)= H_i= 0$ for all $i\in \{1,\ldots, k\}$. Conversely, if $H_i=0$ for all $1\leq i\leq k$, then we can take $\sigma=\id$ as well as $l$ and all $\lambda_i$ and $m_i$ to be zero in \eqref{eq:condition for translate equal to tensor with line bundle}, showing that $E$ is homogeneous.

If $E$ is semi-homogeneous, we take differences of the equations \eqref{eq:condition for translate equal to tensor with line bundle} for given  $1\leq i,j\leq n$. For every $x \in N_\R$ there exists a permutation $\sigma$ of $\{1,\ldots, k\}$ and elements $\lambda\in \Lambda$ and $m\in M$ with
\[
l_i-l_j +\big(H_j(x)-H_i(x)\big)= l_{\sigma(i)}-l_{\sigma(j)}+ H(\lambda)+m.
\]
As there is only a countable set of choices for the triple $(\sigma, \lambda, m)$, there exists a set $S\subseteq N_\R$ whose affine  span is $N_\R$ such that $H_j(s)-H_i(s)$ is independent of the choice of $s\in S$. But this implies that 
\[
\delta(E_i)=H_i=H_j=\delta(E_j).
\]
Conversely, if all $H_i$ coincide, then we can take $\sigma=\id$ as well as $l=-H_1(x)$, and all $\lambda_i$ and $m_i$ equal to zero in \eqref{eq:condition for translate equal to tensor with line bundle} and conclude that $E$ is semi-homogeneous.
\end{proof}

\subsection{A parameter space for  semi-homogeneous tropical bundles}

\begin{lemma}
    \label{lem:pull-back is isomorphism}
    Let $f\colon \Sigma'=N_\R/\Lambda_1\to N_\R/\Lambda_2=\Sigma$, where $\Lambda_1$ is a finite-index sublattice of $\Lambda_2$. Then the pull-back morphism
    \[
        f^*\colon \Pic^0(N_\R/\Lambda_2)\longrightarrow \Pic^0(N_\R/\Lambda_1)
    \]
is a bijection.
\end{lemma}

We reiterate that $f^\ast$ is only a bijective homomorphism; it does not induce an isomorphism of the integral affine structures on the dual tori. 

\begin{proof}[Proof of \ref{lem:pull-back is isomorphism}]
    We have $\Pic^0(N_\R/\Lambda_i)= \Hom(\Lambda_i,\R)/M$ and the pullback is induced by the inclusion $\Lambda_1\to \Lambda_2$. Since its dual $\Hom(\Lambda_2,\R)\to \Hom(\Lambda_1,\R)$ is a bijection that is the identity on $M$, the assertion follows.
\end{proof}

\begin{lemma}
    \label{lem:isomorphisms of different push-forwards}
    Let $f\colon \Sigma'\to \Sigma$ and $g\colon \Sigma''\to \Sigma'$ be connected finite covers of tropical abelian varieties and let $L$ and $M$ be two tropical line bundles on $\Sigma'$. Then $f_*L\cong f_*M$ if and only if $(f\circ g)_* g^*L \cong (f\circ g)_* g^*M$.
\end{lemma}

\begin{proof}
We have $(f\circ g)_*g^*L\cong (f\circ g)_*g^*M$ if and only if there exists an automorphism $\phi$ of $f\circ g$ such that $\phi^*g^* L\cong g^* M$. Similarly, we have $f_*L\cong f_*M$ if and only if there exists an automorphim $\psi$ of $f$ such that $\psi^*L\cong M$. Any automorphism on a  connected cover $h$ of $\Sigma$ is induced by an element $\lambda\in \pi_1(\Sigma,0)$; we denote the automorphism induced by $\lambda$ by  $\phi_{h,\lambda}$. So we have $(f\circ g)_*g^*L\cong (f\circ g)_*g^*M$ if and only if there exists $\lambda\in \pi_1(A,0)$ with $\phi_{f\circ g,\lambda}^*g^L\cong g^*M$. Since $g\circ \phi_{f\circ g,\lambda}= \phi_{f,\lambda}\circ g$, this is equivalent to the existence of $\lambda$ with $\phi_{f,\lambda}^*L\cong M$ by Lemma \ref{lem:pull-back is isomorphism}. This, in turn, is equivalent to $f_*L\cong f_*M$.
\end{proof}

\begin{definition} \label{def:equivalence}
     Let $\Sigma=N_\R/\Lambda$  and $\Gamma$ be a finite-index sublattice of $\Lambda$.
      \begin{enumerate}[(i)]
    \item  An indecomposable tropical vector bundle $E= E(f\colon \Sigma'\to \Sigma,L)$  is called \emph{$\Gamma$-compatible} if the cover $N_\R / \Gamma \to \Sigma$ factors as $f \circ g$ for some $g: N_\R / \Gamma \to \Sigma'$.
  \item
    We say that two indecomposable tropical vector bundles $E_1= E(f_1\colon \Sigma'_1\to \Sigma,L_1)$ and $E_2=E(f_2\colon \Sigma'_2\to \Sigma,L_2)$ on the tropical abelian variety $\Sigma=N_\R/\Lambda$ are \emph{equivalent}, written $E_1\sim E_2$, if there exist a connected free cover $h\colon \Sigma''\to \Sigma$ and factorizations $h=f_1\circ g_1= f_2\circ g_2$ such that $h_*g_1^*L_1\cong h_*g_2^*L_2$.
    \item We denote by  $M^{\Gamma}_{H,1}(\Sigma)$ the set of equivalence classes of indecomposable $\Gamma$-compatible tropical vector bundles of slope $H$.
    \end{enumerate}
\end{definition}

\begin{remark}\label{rem:equivalence of vector bundles}
    In Definition \ref{def:equivalence} (ii), there exists a connected finite cover $h\colon \Sigma''\to \Sigma$ with the desired properties if and only if all connected finite covers that dominate both $f_1$ and $f_2$ have the desired property. This follows directly from Lemma \ref{lem:isomorphisms of different push-forwards} and the fact that connected finite covers form a directed set. 
\end{remark}

Since by Lemma \ref{lem:decomposition into indecomposable}, any tropical vector bundle on an integral affine manifold can be uniquely expressed as a sum of indecomposable ones, we can extend the Definition \ref{def:equivalence} to all tropical vector bundles.

\begin{definition} \label{def:equivalence of vector bundles}
Let $\Sigma=N_\R/\Lambda$ and $\Gamma$ be a finite-index sublattice of $\Lambda$. 
\begin{enumerate}[(i)]
\item We say that a tropical vector bundle $E$ on $\Sigma$ is \emph{$\Gamma$-compatible} if each of its indecomposable summands is $\Gamma$-compatible.
\item We write $M^{\Gamma}_{H,k}(\Sigma)$ for the set theoretic symmetric power $\Sym^k\big(M_{H,1}^\Gamma(\Sigma)\big)=\left(M_{H,1}^\Gamma(\Sigma)\right)^k/S_k$. We may interpret an element in $M^{\Gamma}_{H,k}(\Sigma)$ as an equivalence class of vector bundles arising as a direct sum $E=E_1\oplus \cdots\oplus E_k$ of indecomposable $\Gamma$-compatible semi-homogeneous vector bundles of slope $H$, where two bundles $E=E_1\oplus \cdots\oplus E_k$ and $E'=E_1'\oplus \cdots \oplus E_k'$ are equivalent if and only if (possibly after a suitable permutation) the $E_i$ and $E_i'$ are equivalent.
\end{enumerate}
\end{definition}

\begin{remark}
     Note that in Definition \ref{def:equivalence of vector bundles}, it is possible that two tropical vector bundles of different rank are equivalent.
\end{remark}

\begin{proposition}
\label{prop:Moduli of Gamma-compatible tropical vector bundles}
    The subgroup $M'=M+H(\Lambda)$ of $\Hom(\Gamma,\R)$, where we consider $H$ as a map from $\Lambda$ to $\Hom_\R(N_\R,\R)\cong \Hom(\Gamma,\R)$, contains $M$ as a finite-index lattice. The set $\Hom(\Gamma, \R)/M'$ can be naturally identified with the set $M^{\Gamma}_{H,1}(\Sigma)$.
\end{proposition}

In particular, the moduli space $M_{H,1}^{\Gamma}(\Sigma)$ carries the structure of a real torus with integral structure. 
Thus the moduli space $M_{H,k}^{\Gamma}(\Sigma)=\Sym^k\big(M_{H,1}^{\Gamma}(\Sigma)\big)$ is a finite quotient of an integral affine manifold.

\begin{proof}[Proof of Proposition \ref{prop:Moduli of Gamma-compatible tropical vector bundles}]

    All underlying covers of indecomposable $\Gamma$-compatible tropical vector bundles can be dominated by the cover $\pi\colon N_\R/\Gamma\to N_\R/\Lambda$, so by Lemma \ref{lem:isomorphisms of different push-forwards} (see also Remark \ref{rem:equivalence of vector bundles}) the set of equivalence classes of $\Gamma$-compatible indecomposable tropical vector bundles on $N_\R/\Lambda$ of slope $H$ is in natural bijection with the set of isomorphism classes of tropical vector bundles that appear as push-forwards of tropical line bundles of Néron-Severi class $H$ on $N_\R/\Gamma$ along $\pi$. This, in turn, is in bijection with the dual abelian variety of $N_\R/\Gamma$ modulo the automorphism group $\Lambda/\Gamma$ of the cover, that is with
    \[
    \big(\Hom(\Gamma,\R)/M\big)\big/(\Lambda/\Gamma) \ .
    \]
    A class $\overline \lambda\in \Lambda/\Gamma$ acts on $N_\R/\Gamma$ as translation by $\lambda$. On the level of factors of automorphy we  have described the action of $\lambda$ on tropical line bundles of Néron-Severi class $H$ by 
    \[
    (H,l)\longmapsto \big(H, l-H(\lambda)\big) \ .
    \]
    This allows us to identify the quotient $\big(\Hom(\Gamma,\R)/M\big)\big/(\Lambda/\Gamma)$ with $\Hom(\Gamma, \R)/M'$.
\end{proof}


\section{Uniformization and semi-homogeneous bundles}\label{section_uniform+semihom}


Let $A^{\an}=T^{\an}/\Lambda$ be the analytification of an abelian variety $A$ with totally degenerate reduction over an algebraically closed complete non-Archimedean field $K$. Let $M$ and $N$ denote the character and cocharacter lattice of $T$, so that  $T=N\otimes_\Z\G_m$. As recalled in \S~\ref{subsec:results on semi-homogeneous bundles}, in \cite{Mukai}, Mukai showed that  for every $H\in \NS(A)_\Q$ there exists a simple semi-homogeneous vector bundle $E$ with $\delta(E)=H$, and $E'$ is another such vector bundle if and only if $E'\cong E\otimes L$ for some $L\in \Pic^0(A)$. He also proved that there exists an isogeny $f\colon B\to A$ and a line bundle $L$ on $B$ such that $f_*L\cong E$. However, neither the isogeny nor the line bundle are unique in general. Moreover, to be able to tropicalize $E$ we need $f$ to be a \emph{free cover}, by which we mean that the associated tropical cover $f^\trop\colon B^\trop\to A^\trop$ is free. In what follows we show that every simple semi-homogeneous bundle can indeed be written as the push-forward of a line bundle along a free cover.

\subsection{Semi-homogeneous vector bundles and free covers}

Using the given uniformization of $A$, we observe that the free covers of $A$ are precisely those given by a quotient map $N\otimes_\Z\G_m^\an/\Lambda'\to  N\otimes_\Z\G_m^\an/\Lambda=A$ for some finite-index sublattice $\Lambda'$ of $\Lambda$. To this end, we make the following definition:

\begin{definition}
    Let $H\in \NS(A)_\Q$. A finite-index sublattice $\Lambda'$ of $\Lambda$ is \emph{$H$-admissible} if there exists a line bundle $L$ on $T^\an/\Lambda'$ whose push-forward to $A$ under the quotient map is simple of slope $H$. We say that $L$ \emph{represents} $H$.
\end{definition}

The following proposition shows that if a line bundle $L$ on $T^\an/\Lambda'$ represents $H$, then its Néron-Severi class $[L]$ is given by $H$ again. 

\begin{proposition}
\label{prop:K(L) for line bundle on cover}
Let $A^{\an}=T^{\an}/\Lambda$ and  $H\in \NS(A)_\Q$. Let $f\colon T^\an/\Lambda'\to A^\an$ be an isogeny, where $\Lambda'$ is a finite-index sublattice of $\Lambda$, and let $L$ be a line bundle on $T^\an/\Lambda'$ such that $\delta(f_*L)=H$. Then $\Lambda'$ is contained in the subgroup $\LargeLattice{H}$ defined in Definition \ref{definition: Large lattice {H}}. Moreover, we have 
\[
\ker(f)\cap K(L)= \big\{\lambda\in \LargeLattice{H} \ \big\vert\  [\lambda, \lambda']_H =[\lambda', \lambda]_H \text{ for all }\lambda'\in \Lambda'\big\}/\Lambda' \ .
\]
In particular, the sublattice $\Lambda'$ is $H$-admissible if and only if it is maximal among the sublattices of $\LargeLattice{H}$ on which $H$ is symmetric.
\end{proposition}

\begin{proof}
\begin{equation}
\label{eq:translate isomorphic to line bundle in terms of factors}
\langle \lambda' , m\rangle =\big\langle x, H(\lambda')\big\rangle
\end{equation}
for all $\lambda'\in \Lambda'$.
The kernel of $f$ is given by $\Lambda/\Lambda'$. So if $\lambda\in \LargeLattice{H}$ with $[\lambda,\lambda']_H=[\lambda',\lambda]_H$ for all $\lambda'\in \Lambda$, then we can take $x=\lambda $ and $m=H(\lambda)$ in \eqref{eq:translate isomorphic to line bundle in terms of factors} and see that the class $\overline \lambda$ in $T^\an/\Lambda'$ satisfies $\overline \lambda\in \ker(f)\cap K(L)$.

Conversely, let $\lambda \in \Lambda$ with $\overline \lambda\in \ker(f)\cap K(L)$. Then by virtue of \eqref{eq:translate isomorphic to line bundle in terms of factors} there exists $m\in M$ such that 
\[
\langle \lambda' , m\rangle =\big\langle \lambda, H(\lambda')\big\rangle 
\]
for all $\lambda'\in \Lambda'$. Because $H\in \NS(A)_\Q$, for a sufficiently large integer $k$ we have 
\[
\langle \lambda', km\rangle =  \big\langle \lambda, kH(\lambda')\big\rangle  
=\big\langle \lambda', kH(\lambda)\big\rangle 
\]
for all $\lambda'\in \Lambda'$. Applying the valuation on both sides and using that $\R$ is torsion-free allows us to conclude that $H(\lambda)= m\in M$, from which we conclude that $\lambda \in \LargeLattice{H}$.
Plugging in $H(\lambda)$ for $m$ above, we see that
\[
[\lambda',\lambda]_H=\big\langle \lambda', H(\lambda)\big\rangle =\big\langle \lambda, H(\lambda')\big\rangle =[\lambda,\lambda']_H
\]
for all $\lambda'\in \Lambda'$.

By our description of $\ker(f)\cap K(L)$,  we have $\ker(f)\cap K(L)=0$ if and only if there exists no sublattice of $\LargeLattice{H}$ that is strictly larger than $\Lambda'$ and on which $H$ defines a symmetric bilinear form. On the other hand, by \cite[Proposition 5.6]{Mukai}, the vector bundle $f_*L$ is simple if and only if $\ker(f)\cap K(L)=0$, completing the proof.
\end{proof}

Recall the definition of the subgroup $\SmallLattice{H}$ given in Definition $\ref{smalllattice{H}}$.

 \begin{proposition}
 \label{prop:writing vector bundles as push-forwards}
     Let $H\in \NS(A)_\Q$. 
     Then there exists an $H$-admissible sublattice $\Lambda'$ of $\Lambda$. Moreover, given $\theta\in \LargeLattice{H}\setminus \SmallLattice{H}$, we may choose $\Lambda'$ such that $\theta\notin \Lambda'$.
 \end{proposition}

 \begin{proof}
 If $\SmallLattice{H} = \LargeLattice{H}$, we can take $\Lambda'= \LargeLattice{H}$. Otherwise, 
 by definition of $\SmallLattice{H}$, there exists $\delta\in \LargeLattice{H}$ such that $[\delta,\theta]_H\neq [\theta,\delta]_H$. Moreover, the restriction of $H$ to $\SmallLattice{H}+\Z\delta$ is symmetric and by Lemma \ref{lem:Gamma contains lLambda}, $\SmallLattice{H}+\Z\delta$ has finite index in $\Lambda$. Let $\Lambda'$ be maximal among the sublattices of $\Lambda$ that contain $\SmallLattice{H}+\Z\delta$ and on which $H$ is symmetric. Let $L$ be a line bundle on $B\coloneqq N\otimes_\Z\G_m^\an/\Lambda'$ of Néron-Severi class $H$, and let $f\colon  B\to A$ denote the quotient map. By construction, we have $\delta(f_*L)=H$, and by Proposition \ref{prop:K(L) for line bundle on cover} it follows that $\Lambda'$ is $H$-admissible. 
Since $H$ is symmetric on $\Lambda'$ and $\delta\in \Lambda'$, we must have $\theta\notin \Lambda'$. 
 \end{proof}

\begin{example}\label{ex:no maximal sublattice on which symmetric}
In Example \ref{ex:algebraic NS not saturated}, the class $H\in \NS(A^\trop)$ (notation as in the example), is not contained in $\NS(A)$. Recall that
\begin{equation}
\label{eq:no maximal sublattice again}
[\lambda_1, \lambda_2]_H =1 \qquad \text{as well as} \qquad 
[\lambda_2,\lambda_1]_H =-1 \ .
\end{equation}
So $H$ is not $\G_m$-symmetric. But $2H$ is $\G_m$-symmetric as $[\lambda_2,\lambda_1]_{2H}=(-1)^2=1$. Therefore $H$ defines an element in $\NS(A)_\Q$. Moreover, if $\Lambda'$ is any index-$2$ sublattice of $\Lambda$ (notation as in the example), then the induced morphism $\Lambda'\to  M$ is $\G_m$-symmetric (e.g.\ if $\Lambda'=2\Z\lambda_1+\Z\lambda_2$, then both lines of \eqref{eq:no maximal sublattice again} are squared, making the form symmetric). In particular, there exists no \emph{maximal} sublattice of $\Lambda$ on which $H$ becomes $\G_m$-symmetric. 
\end{example}

\begin{corollary}
    Let $H\in \NS(A)_\Q$. Then we have
    \[
    \SmallLattice{H}= \bigcap_{\substack{\Lambda' \text{ is }\\ H\textrm{-admissible}} }\Lambda' \ .
    \]
    Moreover, for any $H$-admissible $\Lambda'\subseteq \Lambda$,
        \[ \SmallLattice{H}\subseteq \Lambda'\subseteq \LargeLattice{H}.\]
        
\end{corollary}

\begin{proof}
By Proposition \ref{prop:K(L) for line bundle on cover}, the lattice $\SmallLattice{H}$ is contained in the intersection, and by Proposition \ref{prop:writing vector bundles as push-forwards} the opposite inclusion holds as well. 
\end{proof}

\subsection{Uniformization of $M_{H,1}(A)$}

We have seen so far that simple semi-homogeneous bundles can be written as push-forwards of line bundles along free covers, but (in general) not uniquely. The following proposition controls the non-uniqueness.

\begin{proposition}
\label{prop:line bundle mod SmallLattice from simple vector bundle}
Let $H\in \NS(A)_\Q$.
For $i\in \{1,2\}$, let $\Lambda_i$ be an $H$-admissible sublattice of $\Lambda$ and let $L_i$ be a line bundle on $T^\an/\Lambda_i$ representing $H$ with $f_{1*}L_1\cong  f_{2*}L_2$, where $f_i\colon T^\an/\Lambda_i\to A$ denotes the projection.
Let $g_i\colon T^\an/\SmallLattice{H}\to T^\an/\Lambda_i$ denote the projection. Then there exists $\lambda\in \Lambda$ such that
\[
g_1^*L_1\cong T_{\overline\lambda}^*g_2^*L_2.
\]
\end{proposition}

\begin{proof}
Since $f_i$ is a $\Lambda/\Lambda_i$-torsor, we have
    \[
    f_i^*f_{i*}L_i\cong \bigoplus_{\overline\lambda \in \Lambda/\Lambda_i} T_{\overline \lambda}^*L_i \ .
    \]
 Let $\lambda_1,\ldots,\lambda_k\in \Lambda$ be a set of representatives for $\Lambda/\Lambda_2$, and let $h\colon T^\an/\SmallLattice{H} \to A$ denote the projection. Then we have
    \[
    \bigoplus_{\overline\lambda \in \Lambda/\Lambda_1} g_1^*T_{\overline \lambda}^*L_1 
    \cong h^*f_{1*}L_1 
    \cong h^*f_{2*}L_2 
    \cong \bigoplus_{\overline\lambda \in \Lambda/\Lambda_2} g_2^*T_{\overline\lambda}^*L_2 \ .
    \cong \bigoplus_{i=1}^k T_{\overline \lambda_i}^*g_2^*L_2 \ .
    \]
    Therefore, we have $T_{\overline\lambda_i}^*g_2^*L_2 \cong g_1^*L_1$ for some $1\leq i\leq k$ by the Krull-Schmidt Theorem. 
\end{proof}

\begin{nota}
    As above, let $\mc S_H$ denote the set of isomorphism classes of simple vector bundles  on $A$ of slope $H$. By Proposition \ref{prop:line bundle mod SmallLattice from simple vector bundle} we obtain a well-defined map 
\[
\chi_H\colon \mc S_H \longrightarrow \Pic^H\big(T^\an\big/\SmallLattice{H}\big)/\Lambda \,
\]
where $\Lambda$ acts by pull-backs via translations. By Lemma \ref{lem:quotient of picard by Galois group as simple vector bundles} we have
\[
\psi_H\colon \mc S_H\xlongrightarrow{\sim} \Pic^H(A_H) \ ,
\]
where $A_H=\big(A^\vee/\Sigma(H)\big)^\vee$. In what follows we identify  $\Pic^H\big(T^\an/\SmallLattice{H}\big)\big/\Lambda$ with $\Pic^H(A_H)$.

\end{nota}

\begin{lemma}
    \label{lem:quotient by large lattice}
    Let $\Lambda'$ be a lattice between $\SmallLattice{H}$ and $\LargeLattice{H}$. Then pull-back induces an isomorphism
    \[
    \Pic^H(T^\an/\Lambda')\big/\LargeLattice{H}\cong \Pic^H(T^\an/\SmallLattice{H}).
    \]
\end{lemma}

\begin{proof}
   For $\lambda\in \LargeLattice{H}$ and $L(H,r)\in \Pic^H(N\otimes_\Z\G_m^\an)(K)/\Lambda'$, the translate $T_{\overline\lambda}^*L(H,r)$ is given by $L(H,r')$, where
   \[
    \frac{r'}{r}(\lambda')=[\lambda,\lambda']_H= \frac{[\lambda,\lambda']_H}{[\lambda',\lambda]_H}[\lambda',\lambda]_H \ .
   \]
   
   Since $[\lambda',\lambda]_H= \big\langle \lambda', H(\lambda)\big\rangle $ and $H(\lambda)\in M$, we have
   \[
    T_{\overline\lambda}^*L(H,r)= L\big(H,B(\lambda,-)r\big) \ ,
   \]
   where we denote
   \[
    B\colon \LargeLattice{H}\times\LargeLattice{H}\longrightarrow \G_m(K),\; \; (\lambda_1,\lambda_2)\longmapsto \frac{[\lambda_1,\lambda_2]_H}{[\lambda_2,\lambda_1]_H} \ .
   \]
   Note that $B(\lambda_1,\lambda_2)=B(\lambda_2,\lambda_1)^{-1}$ and that $B(\lambda_1,\lambda_2)=1$ for all $\lambda_2\in \LargeLattice{H}$ if and only if  $\lambda_2\in \SmallLattice{H}$. In particular, the bilinear pairing $B$ induces an injection
   \[
    \LargeLattice{H}/\SmallLattice{H}\longrightarrow \Hom\big(\LargeLattice{H}/\SmallLattice{H},\G_m(K)\big) \ ,
   \]
   which is in fact a bijection because both domain and target have the same cardinality. Since $\G_m(K)$ is divisible, we conclude that the linear maps $\Lambda'\to \G_m(K)$ arising as $B(\lambda,-)$ for some $\lambda\in \LargeLattice{H}$ are precisely those that are identically $1$ on $\SmallLattice{H}$. These in turn are exactly those that are in the kernel of the isogeny
   \[
    \Hom(\Lambda',\G_m^\an)\longrightarrow \Hom(\SmallLattice{H},\G_m^\an)
   \]
   of tori that induces the pull-back of line bundles.   
\end{proof}

It remains to understand how the action of $\Lambda/\LargeLattice{H}$ changes $\Pic^H(T^\an/\SmallLattice{H})$. It turns out that this quotient does not further affect the lattice $\SmallLattice{H}$ that we divide by, but instead changes the cocharacter lattice $N$. 

\begin{nota}
  In what follows, we denote by $H^\vee\colon N_\Q\to \Lambda^\vee_\Q$, the dual of $H$.  
\end{nota}

\begin{definition}
    We define 
    \[
    \NLargeLattice{H}=\big\{ n\in N \ \big\vert \ H^\vee(n)\in \Lambda^\vee \big\} 
    \]
    and $\MLargeLattice{H}=\Hom(\NLargeLattice{H},\Z)$.
\end{definition}

\begin{lemma}
\label{lem:short exact sequence for large M}
    We have
    \[
    \MLargeLattice{H}=M+H(\Lambda) \ .
    \]
    More precisely, there is a short exact sequence
    \[
    0 \longrightarrow \LargeLattice{H} \longrightarrow M\oplus \Lambda \longrightarrow \MLargeLattice{H} \longrightarrow 0 \ ,
    \]
    where the first map sends $\lambda$ to $\big(-H(\lambda), \lambda\big)$, and the second sends $(m,\lambda)$ to $\big(m+H(\lambda)\big)\big\vert_{\NLargeLattice H}$.
\end{lemma}

\begin{proof}
    By definition of $\LargeLattice{H}$, there is an injection
    \[
    \Lambda/\LargeLattice{H}\longrightarrow \MLargeLattice{H}/M ,\;\; \overline\lambda\longmapsto \overline{H(\lambda)}.
    \]
    On the other hand, by definition of $\NLargeLattice{H}$, we also have an injection
    \[
    N/\NLargeLattice{H}\longrightarrow \Hom(\LargeLattice{H},\Z)\big/\Lambda^\vee,\;\; \overline n\longmapsto \overline{H^\vee(n)}.
    \]
    Since 
    \[
    [\MLargeLattice{H}:M]=[N:\NLargeLattice{H}] \qquad\text{and}\qquad [\Lambda:\LargeLattice{H}]=\big[\Hom(\LargeLattice{H},\Z):\Lambda^\vee\big] \ ,
    \] 
    it follows that both of these injections are in fact bijections, proving the first part of the assertion and the exactness of the sequence to the right. The exactness on the left is clear and exactness in the middle follows from the definition of $\LargeLattice{H}$.
\end{proof}

\begin{lemma}
\label{lem:extending the pairing}
    There exists a unique pairing 
    \[
    B\colon\SmallLattice{H}\times \MLargeLattice{H}\longrightarrow \G_m(K)
    \]
    such that $B(\lambda,m)= \langle \lambda, m\rangle$ and $B\big(\lambda,H(\lambda')\big)= [\lambda',\lambda]_H$ for all $\lambda\in \SmallLattice{H}$, $m\in M$, and $\lambda'\in \Lambda$.
\end{lemma}

\begin{proof}
By the exactness of the short exact sequence in Lemma \ref{lem:short exact sequence for large M}, the assertion follows from the fact that for $\lambda\in \SmallLattice{H}$ and $\lambda'\in\LargeLattice{H}$ we have
\[
[\lambda',\lambda]_H=[\lambda,\lambda']_H=\big\langle \lambda, H(\lambda')\big\rangle
\]
\end{proof}

By definition, the unique pairing from Lemma \ref{lem:extending the pairing} extends the pairing $\SmallLattice{H}\times M\to \G_m(K)$ coming from the embedding of $\SmallLattice{H}$ into $T=N\otimes_\Z\G_m^\an(K)$. Therefore, no confusion will arise if we denote it by $\langle -,-\rangle$ as well. Defining this extension of the pairing is equivalent to lifting the embedding of $\SmallLattice{H}$ into $N\otimes_\Z\G_m^\an(K)$ to an embedding  into $\NLargeLattice{H}\otimes_\Z\G_m^\an(K)$.

\begin{nota}
In what follows we write $\NLargeLattice{H}\otimes_\Z\G_m^\an/\SmallLattice{H}$ for the quotient by this particular lift of $\SmallLattice{H}$.
\end{nota}

\begin{lemma}
    \label{lem:description of quotient}
    Let $\Lambda'$ denote a lattice between $\SmallLattice{H}$ and $\LargeLattice{H}$. Then the pull-back induces an isomorphism
    \[
        \Pic^H\big(T^\an/\Lambda' \big)\big/\Lambda \cong \Pic^H\big(\NLargeLattice{H}\otimes_\Z\G_m^\an/\SmallLattice{H}\big)
    \]
\end{lemma}

\begin{proof}
  By Lemma \ref{lem:quotient by large lattice}, we already know that pulling-back to $T^\an=N\otimes_\Z \G_m^\an/\SmallLattice{H}$ is equivalent to dividing by the action of $\LargeLattice{H}$. We may therefore assume that $\Lambda'=\SmallLattice{H}$ and that the restriction of the $\Lambda$-action to $\LargeLattice{H}$ is trivial. 

  Given $\lambda\in \Lambda$ and $L(H,r)\in \Pic^H(N\otimes_\Z\G_m^\an/\SmallLattice{H})$, the translate $T^*_{\overline\lambda}L(H,r)$ is given by $L(H,r')$, where we have
  \[
    \frac{r'}{r}(\lambda')= [\lambda,\lambda']_H= \big\langle \lambda',H(\lambda)\big\rangle
  \]
  for all $\lambda'\in \SmallLattice{H}$. Here, the second equality uses the definition of the embedding of $\SmallLattice{H}$ in $\NLargeLattice{H}$, or dually the embedding of $\MLargeLattice{H}$ in $\Hom(\SmallLattice{H},\G_m^\an)$. By Lemma \ref{lem:short exact sequence for large M}, the given action of $\Lambda$ yields the same quotient as the action by $\MLargeLattice{H}$, that is
  \[
    \Pic^H\big(T^\an/\SmallLattice{H}\big)\big/\Lambda \cong \big(\Hom(\SmallLattice{H},\G_m^\an)\big/M\big)/\Lambda \cong \Hom\big(\SmallLattice{H},\G_m^\an\big)\big/\MLargeLattice{H}\cong \Pic^H\big(\NLargeLattice{H}\otimes_\Z\G_m^\an\big/\SmallLattice{H}\big) \ ,
  \]
  and it is immediate that the quotient map is precisely the pull-back map.
\end{proof}

By Lemma \ref{lem:description of quotient}, we may view $\chi_H$ as a map from $\mc S_H$ to $\Pic^H(\NLargeLattice{H}\otimes_\Z\G_m^\an/\SmallLattice{H})$.

\begin{proposition}
\label{prop:description of S_H}
The map
\[
\chi_H\colon \mc S_H\longrightarrow \Pic^H(\NLargeLattice{H}\otimes_\Z\G_m^\an/\SmallLattice{H})(K)
\]
is a bijection. In particular, we have 
\[
A_H\cong \big(\NLargeLattice{H}\otimes_\Z\G_m^\an\big)\big/\SmallLattice{H} \ ,
\]  
or dually
\[
A^\vee/\Sigma(H) \cong (\SmallLattice{H})^\vee\otimes_\Z\G_m^\an\big/\MLargeLattice{H} \ .
\]
\end{proposition}

\begin{proof}
    By Proposition \ref{prop:writing vector bundles as push-forwards}, there exists an $H$-admissible  sublattice  $\Lambda'$ of $\Lambda$ and a line bundle $L$ on $A'=T^\an/\Lambda'$ such that, if $f\colon A'\to A$ denotes the quotient map, the vector bundle $f_*L$ is simple and $\delta(f_*L)=H$. Given any simple vector bundle $E$ on $A$ with $\delta(E)=H$, there exists a line bundle $M\in \Pic^0(A)$ such that $E=f_*L\otimes M= f_*(L\otimes f^* M)$ by \cite[Proposition 6.17]{Mukai}. Therefore, the simple vector bundle on $A$ of slope $H$ are precisely the push-forwards along $f$ of line bundles in $\Pic^{f^*H}(A')$. If $M_1, M_2\in \Pic^{f^*H}(A')$ with $f_*M_1=f_*M_2$, then $f^*f_*M_1=f^*f_*M_2$. As $f$ is a $\Lambda/\Lambda'$-torsor, we have 
    \[
        f^*f_*M_i= \bigoplus_{\lambda\in \Lambda/\Lambda'} T^*_\lambda M_i
    \]
    and it follows that $M_2\cong T^*_\lambda M_1$ for some $\lambda\in \Lambda/\Lambda'$. Conversely, if $M_2\cong T^*_\lambda M_1$ for some $\lambda\in \Lambda/\Lambda'$, it is clear that $f_*M_1\cong f_*M_2$. We conclude that we can identify $\mc S_H$ with $\Pic^H\big(T^\an/\Lambda'\big)\big/\Lambda$. By Lemma \ref{lem:description of quotient}, the pull-back of line bundles in $\Pic^{f^*H}(A')$ to $\NLargeLattice{H}\otimes_\Z\G_m^\an/\SmallLattice{H}$ induces an isomorphism of $\Pic^{f^*H}(A')/\Lambda$ with $\Pic^H\big(\NLargeLattice{H}\otimes_\Z\G_m^\an/\SmallLattice{H}\big)$. This pull-back factors through the pull-back to $\Pic^H\big(T^\an/\SmallLattice{H}\big)/\Lambda$, and hence through $\chi_H$. 

    Let $g\colon \NLargeLattice{H}\otimes_\Z\G_m^\an/\SmallLattice{H}\to A$ be the quotient map and let $E\in \mc S_H$. for a line bundle $M\in \Pic^0(A)$ we have $\chi_H(E\otimes M)= \psi_H(E)\otimes g^*M$ and since $E\otimes M_1\cong E\otimes M_2$ for $M_1, M_2\in \Pic^0(A)$, the kernel of $g^*$ is precisely $\Sigma(E)$. 
\end{proof}

\section{Tropicalizing semi-homogeneous vector bundles}\label{section_tropicalization}


 We denote by $A^{\an}:=T^{\an}/\Lambda$ the non-Archimedean uniformization and by $A^{\trop}:=T^{\trop}/\Lambda$ the tropicalization of $A$. In this section we study the process of tropicalization for semi-homogeneous vector bundles on an abelian variety $A$ over $K$ that admits a totally degenerate reduction over its valuation ring $R$. We proceed in three steps: first the case of line bundles, then the case of simple semi-homogeneous vector bundles, and then the general case of semi-homogeneous vector bundles. 

\subsection{The case of line bundles}
\label{subsec:tropicalizing line bundles}

Given a line bundle $L$ on $A^{\an}=N\otimes_\Z\G_m^\an/\Lambda$, we can tropicalize it by tropicalizing factors of automorphy. Namely, if we write $L=L(H,r)$, then $H$ defines a symmetric form on $N_\R$. We can also postcompose $r$ with the valuation and obtain a map $\nu\circ r\colon \Lambda\to \R$. However, $\nu\circ r$ is not linear unless $H=0$. To remedy this, we define $r^\trop$ by
\[
r^\trop(\lambda)= \nu\big(r(\lambda)\big)-\frac 12 [\lambda,\lambda]_H^\R \ ,
\]
which is indeed linear: for $\lambda,\lambda' \in \Lambda$ we have
\begin{multline*}
r^\trop(\lambda+\lambda')= \nu\big(r(\lambda+\lambda')\big)-\frac 12 [\lambda+\lambda',\lambda+\lambda']_H^\R
= \\
=\nu \big(r(\lambda)\big)+ \nu\big(r(\lambda')\big)+[\lambda,\lambda']_H^\R -\frac 12 [\lambda,\lambda]_H^\R-[\lambda,\lambda']_H^\R- \frac 12 [\lambda',\lambda']_H^\R
= r^\trop(\lambda)+r^\trop(\lambda') \ ,    
\end{multline*}
where the first comes from applying the valuation to \eqref{eq:condition on non-Arch factor of automorphy}.

\begin{definition}
We define the \emph{tropicalization}  $L^\trop$ of $L$ to be the tropical line bundle $L^\trop=L(H, r^\trop)$. This is independent of the choice of the factors of automorphy.
\end{definition}

Suppose $K'$ is an algebraically closed non-Archimedean extension of $K$. Then, for $L\in\Pic^H(A)$, the base change $L_{K'}\in\Pic^H(A_{K'})$ is a line bundle on $A_{K'}$ with $A_{K'}^{\an}=T_{K'}^{\an}/\Lambda$. The factors of automorphy are compatible with base change. Hence the tropicalization is invariant under this operation. Therefore we have a well-defined tropicalization map 
\begin{equation*}
\trop\colon \Pic^H(A)^{\an}\longrightarrow \Pic^H(A^{\trop}\big)
\end{equation*}
which will turn out to be uniquely determined by its restriction to the dense subset $\Pic^H(A)\subseteq \Pic^H(A)^{\an}$, once we know it is continuous (a consequence of Theorem \ref{thm:tropicalization of line bundles is retraction} right below). 

The dual abelian variety $\Pic^0(A)$ is an abelian variety with maximally degenerate reduction and admits a strong deformation retraction $\tau\colon \Pic^0(A)^{\an}\rightarrow \Sigma\big(\Pic^0(A)\big)$ onto a closed subset $\Sigma\big(\Pic^0(A)\big)\subseteq \Pic^0(A)^{\an}$ that has the structure of a real torus, its \emph{non-Archimedean skeleton} (see \cite[Section 6.5]{Berkbook} but also \cite{Gubler_abeliantropicalization} for details). The variety $\Pic^H(A)$ is a torsor over $\Pic^0(A)$. Hence, its Berkovich analytification also admits a natural strong deformation retraction $\tau\colon \Pic^H(A)^{\an}\rightarrow \Sigma\big(\Pic^H(A)\big)$ onto a non-Archimedean skeleton $\Sigma\big(\Pic^H(A)\big)\subseteq \Pic^H(A)^{\an}$, which now canonically carries the structure of a torsor over a real torus. We note that the skeleton $\Sigma\big(\Pic^H(A)\big)$ coincides with the essential skeleton of $\Pic^H(A)$ in the sense of \cite{KontsevichSoibelman, MustataNicaise, NicaiseXu, NicaiseXuYu} by \cite[Proposition 4.3.2]{HalleNicaise}.

\begin{theorem}
\label{thm:tropicalization of line bundles is retraction}
    Let $H\in \NS(A)$. Then there exists a unique isomorphism $\phi\colon\Sigma(\Pic^H(A))\xrightarrow{\cong} \Pic^H(A^\trop)$ such that the diagram
        \[
    \begin{tikzcd}
        & \Sigma\big(\Pic^H(A)\big)\arrow[dd,"\phi"] \\
        \Pic^H(A)^\an \arrow[ur,"\tau"] \arrow[dr,"\trop"] &  \\ 
         & \Pic^H(A^\trop) \ ,
    \end{tikzcd}
    \]
    where $\tau$ denotes the retraction map, commutes.
\end{theorem}

\begin{proof}
First consider the case $H=0$. Then we have a uniformization $\Lambda^\vee\otimes_\Z\G_m^\an\to \Pic^0(A)$ in which a point 
\[
r\in (\Lambda^\vee\otimes_\Z\G_m^\an)(L)= \Hom\big(\Lambda, \G_m(L)\big)
\]
over some valued field extension $L/K$ maps to the line bundle $L(0,r)$ by \cite[Theorem 2.7.7]{Luetkebohmert}. Let $r^\trop(\lambda):=\nu(r(\lambda))$ and denote by $L(0,r^\trop)$ the class corresponding  to $r^\trop$ in $\Pic^0(A^\trop)=\Lambda^\vee\otimes_\Z\R /M$. 

The tropicalization maps the line bundle associated to $r$ to the tropical line bundle $L(0, r^\trop)$. However, the map 
\begin{equation*}\begin{split}
\Lambda^\vee\otimes_\Z\G_m^\an /M &\longrightarrow \Lambda^\vee\otimes_\Z\R /M\\ r&\longmapsto \nu\circ r
\end{split}\end{equation*}
has a natural section that identifies $\Pic^0(A^\trop)= \Lambda^\vee\otimes_\Z\R /M$ with the skeleton of $\Pic^0(A)=\Lambda^\vee\otimes_\Z\G_m^\an/M$ and the tropicalization map with the retraction map. 

For general $H$, pick a factor of automorphy $(H,r)$ for a line bundle in $\Pic^H(A)$, and let $(H,r^\trop)$ be its tropicalization. Then tensoring with $L(H,r)$ defines an isomorphism $\Pic^0(A)\to \Pic^H(A)$, and tensoring with $L(H,r^\trop)$ defines an isomorphism $\Pic^0(A^\trop)\to \Pic^H(A^\trop)$. On the level of skeleta, we obtain a chain of isomorphisms
\[
\Sigma\big(\Pic^H(A)\big)\xrightarrow{\otimes L(H,r)^{-1}} \Sigma\big(\Pic^0(A)\big)\xrightarrow{\trop} \Pic^0(A^\trop) \xrightarrow{\otimes L(H,r^\trop)} \Pic^H(A^\trop) \ ,
\]
the composite of which we denote by $\phi$. If for $G\in \NS(A)$ we denote by  $\tau_G\colon \Pic^G(A)^\an\to \Sigma(\Pic^G(A))$ the retraction to the skeleton. Then for $L(H,r')\in \Pic^H(A)(L)$ over some valued field extension $L/K$ we have:
\begin{equation*}\begin{split}
    (\phi\circ\tau_H)\big(L(H,r)\big)&= \trop\left(L\left(0,\frac {r'}r\right)\right)\otimes L\big(H,r^\trop\big) \\&= 
 L(H, r'\circ \nu - r\circ\nu +r^\trop)\\ &= L(H,(r')^\trop)\\&=L(H,r')^\trop  \ .
\end{split}\end{equation*}
Therefore, $\phi\circ \tau_H = \trop$. Since $\tau_H$ is surjective, $\phi$ unique with that property.
\end{proof}

\begin{remark}
    Note that if a line bundle $L$ on $A$ is ample, there is an alternative way to tropicalize $L$. First, one takes the effective divisor $D$ on $A$ corresponding to a section of $L$. The image of $D^\an$ under the retraction to $\Sigma(A)$ has the structure of a tropical hypersurface $D^\trop$ in $A^\trop$ \cite{GublerTropicalAbelian}, which in turn corresponds to a tropical line bundle on $A^\trop$. To see that this line bundle coincides with $L^\trop$, first rigidify $L$ so that $D$ corresponds to a theta function $f$. Then $D^\trop$ is defined by the tropicalization of the theta function $f^\trop$ (in the sense of \cite{FosterRabinoffShokriehSoto}) and $f^\trop$ is a theta function with respect to  the induced rigidification of $L^\trop$ \cite[Theorem 4.10]{FosterRabinoffShokriehSoto}.
\end{remark}

We will need the following result:

\begin{lemma}
\label{lem:pull-backs of line bundles commute with tropicalization}
    Let $\phi\colon N\otimes_\Z\G_m^\an/\Lambda\to N'\otimes_\Z\G_m^\an/\Lambda'$ be a homomorphism of analytic tori induced by a homomorphism $\phi_N\colon N\to N'$ of free abelian groups. Then for every line bundle $L$ on $N'\otimes_\Z\G_m^\an$ we have 
    \[
    (\phi^*L)^\trop=\phi_\trop^*L^\trop \ .
    \]
\end{lemma}

\begin{proof}
Let $L(H,r)$ be a line bundle on $N\otimes_\Z\G_m^\an$. If we denote by $\phi_\Lambda\colon \Lambda\to \Lambda'$ the morphism induced by $\phi_N$, then by \cite[Corollary 6.4.3]{Luetkebohmert} we have

\[
\phi^*L(H,r)= L\big(H\circ (\phi_\Lambda\times \phi_N), r\circ\phi_\Lambda\big) \ .
\]
By a short calculation  we deduce the desired equality.
\end{proof}

\subsection{The case of simple bundles}
Let $A$ be an abelian variety with totally degenerate reduction with $A^\an\cong T^\an/\Lambda$ for some algebraic torus $T$, and let $H\in\NS(A)_\Q$. We now construct a natural tropicalization map for simple semi-homogeneous vector bundles:

\begin{equation*}
\trop\colon M_{H,1}(A)^{\an}\longrightarrow M_{H,1}^{\SmallLattice{H}}(A^{\trop}) \ .
\end{equation*}

 Let $E$ be a simple semi-homogeneous vector bundle on $A$ of slope $\delta(E)=H$. To tropicalize $E$, we use Proposition \ref{prop:writing vector bundles as push-forwards} and write $E=f_*L$ for some line bundle $L$ on the source $B$  of an isogeny $f\colon T^\an/\Lambda'\to T^\an\Lambda$ associated to a finite-index sublattice $\Lambda'$ of $\Lambda$. We then tropicalize $L$ and obtain a line bundle $L^\trop$ on $N_\R/\Lambda'$, which we can push forward to a vector bundle $f^\trop_*L^\trop$ on $N_\R/\Lambda$. This is not well-defined, as $\Lambda'$ and $L$ are not well-defined.

To fix this, we can use $\psi_H$ (see Remark \ref{definitionofpsi}) to tropicalize vector bundles in $\mc S_H$. Given $E\in \mc S_H$, we apply $\psi_H$ and obtain a line bundle $L$ on $N\otimes_\Z\G_m^\an/\SmallLattice{H}$, that is well defined up the action of the automorphism group $\big[\Lambda/\SmallLattice{H}\big]$ of the cover $T^\an/\SmallLattice{H}\to A$. Therefore, its tropicalization $L^\trop$ is a tropical line bundle on $N_\R/\SmallLattice{H}$ that is well-defined up to the action of the automorphism group of the cover $N_\R/\SmallLattice{H}\to N_\R/\Lambda$. By Proposition \ref{prop:Moduli of Gamma-compatible tropical vector bundles}, the tropical line bundle $L^\trop$ defines an equivalence class of $\SmallLattice{H}$-compatible tropical vector bundles of slope $H$ that is independent of all choices. In other words, we obtain a natural map
\[
\trop\colon M_{H,1}(A)(K)\longrightarrow \mc M_{H,1}^{\SmallLattice{H}}(A^{\trop})  
\]
By construction, if a line bundle  $L$ on $T_{K'}^\an/\Lambda'$ represents $H$ and $f\colon T_{K'}^\an/\Lambda'\to A$ denotes the projection, then $\trop(f_*L)$ is the equivalence class of $f^\trop_*L^\trop$.

This construction is naturally compatible with non-Archimedean extensions $K'$ of $K$ and so we obtain the desired tropicalization map for simple semi-homogeneous vector bundles:
\begin{equation*}
\trop\colon M_{H,1}(A)^{\an}\longrightarrow M_{H,1}^{\SmallLattice{H}}(A^{\trop}) \ .
\end{equation*}

We now show that $\trop$ agrees with the retraction $\tau\colon M_{H,1}(A)^{\an}\rightarrow\Sigma(M_{H,1}(A))$ to the skeleton $\Sigma(M_{H,1}(A))$ of $M_{H,1}(A)\cong A^\vee/\Sigma(H)$. To this end, we  first represent the quotient $A^\vee/\Sigma(H)$  explicitly  as $(\Lambda')^\vee\otimes \G_m^\an/M'$, for some finite-index sublattice $\Lambda'$ of $\Lambda$ and some lattice $M'$ containing $M$ as a finite-index sublattice.

\begin{theorem}
\label{thm:skeleton in simple case}
    Let $H\in\NS(A)_\Q$. There exists an isomorphism $\phi\colon\Sigma(S_H)\to \mc M_{H,1}^{\SmallLattice{H}}(A^{\trop})$ such the diagram
    \[
    \begin{tikzcd}
        & \Sigma(S_H)\arrow[dd,"\phi","\sim"'] \\
        M_{H,1}(A)^\an \arrow[ur,"\tau"] \arrow[dr,"\trop"'] &  \\ 
         & M_{H,1}^{\SmallLattice{H}}(A^{\trop})\ ,
    \end{tikzcd}
    \]
    where $\tau$ denotes the retraction to $\Sigma\big(M_{H,1}(A)\big)$, commutes.
\end{theorem}

\begin{proof}
    By Proposition \ref{prop:description of S_H}, we have can identify $M_{H,1}(A)^\an$  with  $\Pic^H(\NLargeLattice{H}\otimes_\Z\G_m^\an/\SmallLattice{H})$. As in the line bundle case treated in Theorem \ref{thm:tropicalization of line bundles is retraction},  we have 
    \[
    \Sigma(M_{H,1}(A))=\Pic^H(\NLargeLattice{H}\otimes_\Z\R/\SmallLattice{H})
    =\Hom(\SmallLattice{H},\R)/\MLargeLattice{H} \ ,
    \]
    and $\tau$ is the tropicalization map of line bundles on $\NLargeLattice{H}\otimes_\Z\G_m^\an/\SmallLattice{H}$ of class $H$. On the other hand, we have
    \[
    M_{H,1}^{\SmallLattice{H}}(A^{\trop})=\Hom(\SmallLattice{H},\R)/M+H(\Lambda)
    \]
    by Proposition \ref{prop:Moduli of Gamma-compatible tropical vector bundles} and
    \[
    M+H(\Lambda)=\MLargeLattice{H}
    \]
    by Lemma \ref{lem:short exact sequence for large M}. Therefore, in the representation of the two spaces, we can take $\phi$ to be the identity. It remains to show that the diagram commutes. Since $M_{H,1}(A)(K)$ is dense in $M_{H,1}(A)^{\an}$ it is enough to consider an $E\in \mc S_H$. Let $\Lambda'$ be $H$-admissible, denote $A'=N\otimes_\Z\G_m^\an/\Lambda'$ and $f\colon A'\to A$ the quotient map, and let $L\in \Pic^{f^*H}(A')$ be a line bundle with $f_*L\cong E$. Then $\trop(E)= \big[f^\trop_*\trop(L)\big]$. If $g\colon \NLargeLattice{H}\otimes_\Z\G_m^\an/\SmallLattice{H}\to A'$ denote the quotient map, then the identification of Proposition \ref{prop:Moduli of Gamma-compatible tropical vector bundles} identifies $\trop(E)$ with $(g^\trop)^*\trop(L)$. Since tropicalization of line bundles commutes with pull-backs by Lemma \ref{lem:pull-backs of line bundles commute with tropicalization}, $\trop(E)$ is given by $\trop(g^*L)$, which we have already noted is equal to $\tau(g^*L)$. As we also have $\psi_H(E)=g^*L$, that is $E$ and $g^*L$ are identified in Proposition \ref{prop:description of S_H}, this completes the proof.
\end{proof}

\subsection{The general case}\label{subsection: tropicalizing general case}

Fix $H\in\NS(A)_\Q$. Recall from Theorem \ref{thm:decomposition of semi-homogeneous vector bundles} that every semi-homogeneous vector bundle $E$ of rank $r$ on $A$ with slope $\delta(H)$ is $S$-equivalent to a direct sum $\oplus_{i=1}^k E_i$ of $k$ simple semi-homogeneous vector bundles $E_i$ with $\delta(E_i)=H$, where $k= \frac{r}{n(H)}$ and the $E_i$, the Jordan-Hölder factors of $E$, are uniquely determined up to permuting the indices. In particular, by Theorem \ref{thm:description of moduli spaces} above, we have $M_{H,k}(A)^\an\cong\Sym^k\big(M_{H,1}(A)\big)^\an$. On the other hand, we have defined $M_{H,k}^{\SmallLattice{H}}(A^{\trop})$ to be symmetric power $\Sym^k(M_{H,1}^{\SmallLattice{H}}(A^{\trop})$ in Definition \ref{def:equivalence of vector bundles} above.
We therefore have a natural tropicalization map 
\begin{equation*}
\trop\colon M_{H,k}(A)^{\an}\longrightarrow M_{H,k}^{\SmallLattice{H}}(A^{\trop})
\end{equation*}
given by taking the $k$-symmetric powers of the tropicalization map $\trop\colon M_{H,1}(A)^{\an}\longrightarrow M_{H,1}^{\SmallLattice{H}}(A^{\trop})$. In other words: To tropicalize a given $E\in M_{H,k}(A)^{\an}$ we can apply our construction for simple semi-homogeneous bundles to the Jordan-Hölder factors of $E$ and obtain a well-defined element in $M_{H,k}^{\SmallLattice{H}}(A^{\trop})=\Sym^k(M^{\SmallLattice{H}}_{H,1}(A^\trop))$.

The moduli space $M_{H,k}(A)$ is Calabi-Yau. Hence it admits a natural strong deformations retraction $\tau$ to the essential skeleton $\Sigma(M_{H,k}(A))$ (see \cite{KontsevichSoibelman} as well as \cite{MustataNicaise, NicaiseXu, NicaiseXuYu} for details on this construction). Since essential skeletons behave well with respect to symmetric products by \cite{BrownMazzon}, we obtain the following theorem:

\begin{theorem}
\label{thm:skeleton in general case}
    Let $H\in\NS(A)_\Q$ and $k\geq 1$. For the moduli space $M_{H,k}(A)$ there exists an isomorphism 
    \[
    \phi\colon \Sigma\big(M_{H,k}(A)\big)\xlongrightarrow{\sim} M_{H,k}^{\SmallLattice{H}}(A^{\trop}) \ ,
    \]
    such that the diagram
    \[
    \begin{tikzcd}
        & \Sigma\big(M_{H,k}(A)\big)\arrow[dd,"\phi","\sim"'] \\
        M_{H,k}(A)^\an \arrow[ur,"\tau"] \arrow[dr,"\trop"'] &  \\ 
         &  M_{H,k}^{\SmallLattice{H}}(A^{\trop}) \ .
    \end{tikzcd}
    \]
\end{theorem}

\begin{proof}
    By \cite[Proposition 6.1.11]{BrownMazzon} the formation of the essential skeleton is compatible with symmetric products. The statement of our theorem is therefore a direct consequence of Theorem \ref{thm:skeleton in simple case}.
\end{proof}


\section{Homogeneous bundles and representations}



Let $A$ be an abelian variety over $K$ admitting a totally degenerate reduction over $R$; we write $A^{\an}=T^{\an}/\Lambda$ for its non-Archimedean uniformization and $A^{\trop}=T^{\trop}/\Lambda$ for its tropicalization. In this section we study a natural analytic epimorphism from the analytification of the character variety $X_r(\Lambda)$ of $\Lambda$ to the analytification of the moduli space $M_{0,r}(A)$ building on the non-Archimedean uniformization of the dual abelian variety. We introduce a tropical analogue of this construction and show that both of these construction are compatible under tropicalization. 

\subsection{The moduli space $M_{0,r}(A^{\trop})$ -- made explicit}\label{section_M0r} Let $M_{0,r}(A^\trop)$ be the set of isomorphism classes of homogeneous vector bundles on $A^{\trop}$ of rank $r$. 

 Every vector bundle $E\in M_{0,r}(A^{\trop})$ may be written as a direct sum $E_1\oplus \cdots\oplus E_s$ of indecomposable vector bundles. Recall that by Proposition \ref{prop_charsemihom} the bundle $E$ is homogeneous if $\delta(E_i)=0$ for all $i=1,\ldots, s$. For every $E_i$ there is a sublattice $\Lambda_i\subseteq \Lambda$ of index $k_i>0$ with $k_1+\cdots + k_s=r$ such that there is a line bundle $L_i\in \Pic(N_\R/\Lambda_i)$ whose pushforward along $N_\R/\Lambda_i\rightarrow N_\R/\Lambda$ is equal to $E_i$. So we may write  $M_{0,r}(A^{\trop})$ naturally as a disjoint union 
\begin{equation*}
M_{0,r}(A^{\trop})=\bigsqcup_{\substack{\Lambda_1\oplus\ldots\oplus \Lambda_s\subseteq \Lambda^s\\ [\Lambda:\Lambda_i]=k_i>0\\k_1+\ldots+k_s=r}} M_{\Lambda_1\oplus\cdots\oplus\Lambda_{s}}(A^{\trop}) \ ,
\end{equation*}
where $M_{\Lambda_1\oplus\cdots\oplus\Lambda_s }(A^{\trop})$ denotes the set of isomorphism classes of vector bundles $E$ that decompose as direct sums $E=E_1\oplus \cdots\oplus E_s$ of indecomposable vector bundles $E_i$ that arise as pushforwards of line bundles along the connected cover $N_\R/\Lambda_i\rightarrow N_\R/\Lambda$ for sublattices $\Lambda_i\subseteq\Lambda$ of index $k_i>0$. When $s=1$, we may identify $M_{\Lambda_1}^{\trop}(\Lambda)$ 
as
\begin{equation*}
M_{\Lambda_1}(A^{\trop})=\Pic^0(N_\R/\Lambda_i)=\Hom(\Lambda_1,\R)/N\simeq\Hom(\Lambda,\R)/N= \Lambda^\vee_\R/N=A^{\vee,\trop} \ .
\end{equation*}
In general, for arbitrary $s\geq 1$, let $l_j=\#\{k_i=j\}$ for $j=1,\ldots, r$. Then the group $S_{l_1}\times\cdots\times S_{l_r}$ permutes summands of equal length in $E=E_1\oplus\cdots\oplus E_s$ and we have an identification 
\begin{equation*}
M_{\Lambda_1\oplus\cdots\oplus\Lambda_s }(A^{\trop})=(A^{\vee,\trop})^s/S_{l_1}\times\cdots\times S_{l_r} \ .
\end{equation*}
We use this identification to endow $M_{0,r}(A^{\trop})$ with the structure of an integral affine orbifold. 

The \emph{main component} of $M_{0,r}(A^{\trop})$ is given by
\begin{equation*}
M_{\underbrace{\Lambda\oplus \cdots \oplus \Lambda}_{r \textrm{ times}}}(A^{\trop})=(\Lambda^\vee_\R/N)^r/S_r=\Sym^r(\widehat{A}^{\trop}) \ .
\end{equation*}
It is naturally isomorphic to the moduli space $M_{0,r}^{\Lambda}(A^{\trop})$.

\subsection{Tropical representations and tropical character variety}\label{section_tropcharvar}

Let $\rho\colon \Lambda\rightarrow \GL_r(\T)$ be a representation. We point out that in $\GL_r(\T)$ (the group of generalized tropical permutation matrices) all block triangular matrices are already block diagonal matrices. Thus, one should not distinguish between indecomposable and irreducible $\GL_r(\T)$-representations. In particular, every representation with values in $\GL_r(\T)$ can be viewed as \emph{semisimple}, since it is a direct sum of indecomposable representations.

We associate to $\rho$ a vector bundle $E(\rho)$ of rank $r$ on $A^{\trop}=N_\R/\Lambda$, whose sections (as an $S_r\ltimes \Aff^r$-torsor) over $U\subseteq A^{\trop}$ may be identified with affine functions $f\colon\widetilde{U}\rightarrow \R^r$ on the preimage $\widetilde{U}$ of $U$ in $N_\R$ such that $f(\lambda + \widetilde{u})=\rho(\lambda) f(\widetilde{u})$ for all $\widetilde{u}\in \widetilde{U}$ and $\lambda\in\Lambda$.

\begin{proposition}\label{prop_NarasimhanSeshadri}
For a representation $\rho\colon \Lambda\rightarrow \GL_r(\T)$ the vector bundle $E(\rho)$ is homogeneous and every homogeneous vector bundle arises this way.  
\end{proposition}

\begin{proof}
We first consider the case $r=1$. For a representation $\rho\colon \Lambda\rightarrow \GL_1(\T)=\R$ the associated line bundle $L(\rho)=L(H,l)$ is defined by the factor of automorphy $H=0$ and $l=\phi\in\Hom(\Lambda,\R)$, equivalently, by the surjective homomorphism $\Lambda^\vee_\R\rightarrow \Lambda^\vee_\R/M=\Pic^0(A^{\trop})$. This proves our claim for $r=1$.

Now suppose that $\rho$ is an indecomposable representation. Note that analogous to Proposition \ref{prop_vectorbundle=cover+linebundle}, $\GL_r(\T)$-torsors (that is tropical vector bundles with constant transition functions) on any integral affine manifold $X$ are in natural one-to-one correspondence with pairs consisting of a degree-$r$ free cover $Y\to X$ and a torsor over $\GL_1(\T)=\R$ on $Y$. Since all covers and $\R$-torsors over the universal covering space $N_\R$ of $A^\trop$ are trivial, it follows that conjugacy classes of indecomposable $\GL_r(\T)$-representations are in one-to-one correspondence with $\GL_r(\T)$-torsors on $A^\trop$, whose associated cover is connected. By Lemma \ref{lem:covers of tropical Abelian variety}, the associated cover is of the form $N_\R/\Lambda'\to N_\R/\Lambda$ for some index-$r$ sublattice $\Lambda'$ of $\Lambda$. Since $\R$-torsors on $N_\R/\Lambda'$ correspond to representations $\rho'\colon \Lambda'\to \R$, this establishes a natural one-to-one correspondence between indecomposable representations $\rho\colon \Lambda\rightarrow\GL_r(\T)$ up to conjugation and pairs consisting of a sublattice $\Lambda'\subseteq \Lambda$ of index $r$ and a representation $\rho'\colon\Lambda'\rightarrow \R$. By construction of the correspondence, the vector bundle $E(\rho)$ is the pushforward along $N_\R/\Lambda'\rightarrow N_\R/\Lambda$ of the line bundle $E(\rho')=L(0,\rho')$ on $N_\R/\Lambda'$. This vector bundle is indecomposable, since the cover $N_\R/\Lambda'\rightarrow N_\R/\Lambda$ is connected. It is of slope $0$ since $E(\rho')$ has slope zero. Moreover, every vector bundle of rank $r$ and slope zero arises as the pushforward of a line bundle of slope zero on a connected cover of degree $r$ of $A^{\trop}$, which one may describe by choosing a sublattice $\Lambda'$ of index $r$. So every indecomposable vector bundle of slope zero arises as $E(\rho)$ for an irreducible representation $\rho\colon\Lambda\rightarrow \GL_r(\T)$. 

Finally, for the general case, we recall from Proposition \ref{prop_charsemihom} that every representation $\rho\colon\Lambda\rightarrow\GL_r(\T)$ is naturally a direct sum of irreducible representations and every homogeneous vector bundle is a direct sum of indecomposable vector bundles of slope zero. Since the association $\rho\mapsto E(\rho)$ is compatible with direct sums, the irreducible case of our claim implies the general one. 
\end{proof}

We now give a moduli-theoretic reinterpretation of Proposition \ref{prop_NarasimhanSeshadri}.

\begin{definition}\label{tropical character variety}
 We define the \emph{tropical character variety} denoted  $X_r^{\trop}(\Lambda)$, as the set of conjugacy classes of representations $\rho\colon\Lambda\rightarrow \GL_r(\T)$.   
\end{definition}

\begin{remark}
    Note that every representation $\rho\colon\Lambda\rightarrow \GL_r(\T)$ may be written as a direct sum $\rho^1\oplus \cdots\oplus\rho^s$ of indecomposable representations and the underlying permutation representation of $\rho^i$ gives rise to a sublattice $\Lambda_i\subseteq \Lambda$ of index $k_i>0$ such that $k_1+\cdots + k_s=r$.
Moreover, any such permutation representation gives rise to a connected cover $N_\R/\Lambda_i\rightarrow N_\R/\Lambda$.  
\end{remark}

 So we can write  $X_r^{\trop}(\Lambda)$ naturally as a disjoint union 
\begin{equation*}
X_r^{\trop}(\Lambda)=\bigsqcup_{\substack{\Lambda_1\oplus\ldots\oplus \Lambda_s\subseteq \Lambda^s\\ [\Lambda:\Lambda_i]=k_i>0\\k_1+\ldots+k_s=r}} X_{\Lambda_1\oplus\cdots\oplus\Lambda_{s}}^{\trop}(\Lambda) \ ,
\end{equation*}

When $s=1$, we have
\begin{equation*}
X_{\Lambda_1}^{\trop}(\Lambda)=\Hom(\Lambda_1,\R)\simeq\Hom(\Lambda,\R)= \Lambda^\vee_\R \ .
\end{equation*}
In general, for arbitrary $s\geq 1$, let $l_j=\#\{k_i=j\}$ for $j=1,\ldots, r$. Then the group $S_{l_1}\times\cdots\times S_{l_r}$ permutes summands of equal length in $\rho=\rho^1\oplus\cdots\oplus\rho^s$ and we have an identification 
\begin{equation*}
X_{\Lambda_1\oplus\cdots\oplus\Lambda_s }^{\trop}(\Lambda)=(\Lambda^\vee_\R)^s/S_{l_1}\times\cdots\times S_{l_r} \ .
\end{equation*}
We use this identification to endow $X_r^{\trop}(\Lambda)$ with the structure of an integral affine orbifold. 

\begin{definition}
The component 
\begin{equation*}
X_{\underbrace{\Lambda\oplus\cdots\oplus \Lambda}_{r \textrm{ times}}}=(\Lambda_\R)^r/S_r
\end{equation*}
of the tropical character variety $X_r^{\trop}(\Lambda)$ parametrizes diagonalizable representations and will be denoted by $X^{\trop}_r(\Lambda)^{\diag}$.
\end{definition}

\begin{corollary}\label{cor_tropicalcorrespondence}
There is a natural surjective morphism $\eta_{A^{\trop}}^r\colon X_r^{\trop}(\Lambda)\rightarrow M_{0,r}(A^{\trop})$ (of integral affine orbifolds) that associates to a representation $\rho\colon \Lambda\rightarrow\GL_r(\T)$ the homogeneous vector bundle $E(\rho)$. 
\end{corollary}

We point out that $\eta_{A^{\trop}}^r$ sends $X^{\trop}_r(\Lambda)^{\diag}$  into $M^\Lambda_{0,r}(A^\trop)$.

\begin{proof}[Proof of Corollary \ref{cor_tropicalcorrespondence}]
This is an immediate consequence of the explicit descriptions of $X_r^{\trop}(\Lambda)$ right above, of $M_{0,r}(A^{\trop})$ in Section \ref{section_M0r}, and Proposition \ref{prop_NarasimhanSeshadri}.
\end{proof}

\subsection{Line bundles and representations via non-Archimedean uniformization} \label{section_charactervariety}
As above, let $A$ be an abelian variety over $K$ such that  $A^{\an}\simeq T^{\an} / \Lambda$, where $T$ is a split algebraic torus with character lattice $M$ and cocharacter lattice $N$, and $\Lambda$ is a lattice. 

\begin{definition}
Let $X_r(\Lambda)$ denote the geometric invariant theory quotient 
\begin{equation*}
X_r(\Lambda)=\Hom(\Lambda,\GL_r)\sslash {\GL_r} \ ,
\end{equation*}
where $\GL_r$ acts by conjugation. We will refer to $X_r(\Lambda)$ as the \emph{character variety} of $\Lambda$.
\end{definition} 

The character variety $X_r(\Lambda)$ is a moduli space for semisimple representations $\rho\colon \Lambda\rightarrow{\GL_r}$ up to conjugation. We refer the reader to \cite{Sikora_charactervarieties} for a careful and detailed exposition of this construction in general. 
Since $\Lambda$ is abelian, every irreducible representation of $\Lambda$ is one-dimensional. Again, we may write a semisimple representation $\rho\colon \Lambda\rightarrow{\GL_r}$ as a direct sum $\rho^1\oplus \cdots \oplus \rho^r$ of one-dimensional representation $\rho^i\colon \Lambda\rightarrow \G_m$. This shows that we have a natural isomorphism
\begin{equation*}
X_r(\Lambda)\simeq \Sym^r \big(X_1(\Lambda)\big)=\Sym^r \big(\Lambda^\vee \otimes \G_m\big) \ .
\end{equation*}

We may naturally associate to a semisimple representation a homogeneous vector bundle. This may be phrased from a moduli-theoretic point of view as follows:

\begin{nota}
In what follows, we denote by $X_r^{\an}(\Lambda)$ the analytification of the character variety $X_r(\Lambda)$.
\end{nota}

\begin{proposition}\label{prop_modularvdPR}

There is a natural surjective analytic morphism $\eta_A^r\colon X_r^{\an}(\Lambda)\rightarrow M_{0,r}(A)^{\an}$ given by associating to a semisimple representation $\rho\colon \Lambda\rightarrow\GL_r(L)$ for a non-Archimedean extension $L$ of $K$, the vector bundle $E(\rho)=L(\rho^1)\oplus\cdots\oplus L(\rho^r)$ on $A_L$. 
\end{proposition}

\begin{proof}
For $r=1$ this is the surjective analytic morphism \begin{equation*}\eta_A^1\colon \Lambda^\vee\otimes \G_m^{\an}\longrightarrow \Lambda^\vee\otimes \G_m^{\an}/N=\Pic^0(A)^{\an}\end{equation*}
defining the non-Archimedean uniformization of $A^\vee=\Pic_0(A)$. The general morphism $\eta_A^r$ is given by setting
\begin{equation*}
\eta_A^r=\Sym^r(\eta_A^1)\colon X_r^{\an}(\Lambda)=\Sym^r\big(\Lambda^\vee\otimes \G_m^{\an}\big)\longrightarrow\Sym^r(A^\vee)=M_{0,r}(A)^{\an}
\end{equation*}
and this clearly defines a surjective analytic morphism. 
\end{proof}

\begin{remark}\label{remark_boundedness}
     If we restrict the surjective morphism of the previous theorem to a fundamental domain for the action of $H$, we can derive a bijection. To be precise, let us fix any $\mathbb{Z}$-basis $\lambda_1, \ldots, \lambda_g$ of $\Lambda$ and a $\mathbb{Z}$-basis $x_1, \ldots, x_g$ of $M$. Since the matrix $$\begin{bmatrix}-\log \big|x_i(\lambda_j)\big|\end{bmatrix}_{i,j=1,\ldots, g}$$ has full rank over $\mathbb{R}$ we can exchange the order of the $\lambda_i$ and the $x_j$ so that the entries on the diagonal are all non-zero. Now we replace $\lambda_j$ by $\lambda_j^{-1}$, whenever $-\log\big|x_j(\lambda_j)\big| > 0$, so that the  new bases satisfy  $\big|x_i(\lambda_i)\big| < 1$ for all $i$. Then every class in $\Hom(\Lambda, K^\ast)$ modulo $M$ has a unique representative $\rho: \Lambda \rightarrow K^\ast$ satisfying 
\begin{equation*}
    \label{eqn:rank1}
\big|x_i(\lambda_i)\big| \leq \big|\rho(\lambda_i)\big| < 1 \qquad \mbox{for all } i = 1, \ldots, g.
\end{equation*}
In this way, homogeneous line bundles on $A$ correspond bijectively to representations $\rho: \Lambda \rightarrow K^\ast$ satisfying the previous condition. Note that 
this condition coincides with $\Phi$-boundedness of one-dimensional representation in \cite{Reversatvanderput}. In  \cite{Reversatvanderput}, Theorem (1.3.5) and Section 2,  an equivalence of categories relating homogeneous vector bundles of rank $r$ on $A$ and $r$-dimensional $\Phi$-bounded representations of the lattice $\Lambda$ is established. 
\end{remark}

\subsection{Tropicalization of character varieties}\label{section_tropcharvarII}

There is a natural tropicalization map 
\begin{equation*}
\trop\colon X_r^{\an}(\Lambda)\longrightarrow X_r^{\trop}(\Lambda)
\end{equation*}
that is given as follows: A point $x$ in $X^{\an}_r(\Lambda)$ may be represented by an $L$-valued point of $X_r(\Lambda)$ for a non-Archimedean extension $L$ of $K$. This, in turn, corresponds to a semisimple representation $\rho_x\colon \Lambda\rightarrow \GL_r(L)$. Write $\rho_x$ as a direct sum of irreducible representations $\rho_x^i\colon \Lambda\rightarrow \G_m$ for $i=1,\ldots, r$ (noting that irreducible representations of an abelian group are automatically one-dimensional). Then $\trop(x)$ is defined to be the tropical  representation $\rho_x^{\trop}=(\rho_x^1)^{\trop}\oplus \cdots \oplus (\rho_x^r)^{\trop}$, where $(\rho_x^i)^{\trop}$ is given by the composition $\Lambda\xrightarrow{\rho^i_x}L^\ast\xrightarrow{-\log\vert .\vert}\R=\GL_1(\T)$. Note that image of the map $\trop$ is contained in the component 
\begin{equation*}
X^{\trop}_r(\Lambda)^{\diag}=\Sym^r\big(\Lambda_\R^{\vee}\big)
\end{equation*}
parametrizing diagonal tropical representations of $\Lambda$.

\begin{proposition}\label{prop_skeletoncharactervariety}
Let $\Lambda$ be a finitely generated free abelian group and $\Sigma\big(X_r(\Lambda)\big)$ be the essential skeleton of the character variety 
$X_r(\Lambda)$. Then there is a natural isomorphism
\begin{equation*}
J\colon X^{\trop}_r(\Lambda)^{\diag} \xlongrightarrow{\sim} \Sigma\big(X_r(\Lambda)\big)
\end{equation*}
which makes the diagram
\begin{equation*}
\begin{tikzcd}
& X_r^{\an}(\Lambda)\arrow[ld,"\tau"']\arrow[rd,"\trop"] & \\
\Sigma\big(X_r(\Lambda)\big) && X^{\trop}_r(\Lambda)^{\diag}\arrow[ll,"\sim"',"J"]
\end{tikzcd}
\end{equation*}
commute. 
\end{proposition}

\begin{proof}
Let us first consider the case $r=1$. Then we have $X_1(\Lambda)=\Lambda^\vee\otimes\G_m$ and $X_\Lambda^{\trop}(\Lambda)=\Lambda^\vee\otimes \R$. Therefore, we have a natural isomorphism $J\colon \Lambda^\vee\otimes\R \xrightarrow{\sim} \Sigma\big(\Lambda^\vee\otimes \G_m^{\an}\big)$ which makes the following diagram 
\begin{equation*}
\begin{tikzcd}
& \Lambda^\vee\otimes\G_m^{\an}\arrow[ld,"\tau"']\arrow[rd,"\trop"] & \\
\Sigma\big(\Lambda^\vee\otimes\G_m^{\an}\big) && \Lambda^\vee\otimes\R\arrow[ll,"\sim"',"J"] 
\end{tikzcd}
\end{equation*}
commute. The general case follows from the compatibility of essential skeletons of open 
Calabi--Yau varieties with symmetric powers (see \cite[Proposition 6.1.11]{BrownMazzon} and \cite[Theorem F]{MMS_geometricP=W}). 
\end{proof}

\begin{proposition}\label{prop_compatibilityrep->bundle}
For a non-Archimedean extension $L\vert K$ and a semisimple representation $\rho\colon \Lambda\rightarrow \GL_r(L)$ we have 
\begin{equation*}
E(\rho)^{\trop}\simeq E(\rho^\trop) \ . 
\end{equation*}
In other words, the diagram 
\begin{equation}\label{eq_correspondencetropicalizes}\begin{tikzcd}
X_r^\an(\Lambda)\arrow[d,"\trop"']\arrow[rr,"E(-)^{\an}"]&& M_{0,r}(A)^{\an}\arrow[d,"\trop"]\\
X_r^\trop(\Lambda)\arrow[rr,"E(-)"] && M_{0,r}(A^{\trop})
\end{tikzcd}\end{equation}
commutes. 
\end{proposition}

\begin{proof}
For $r=1$, diagram \eqref{eq_correspondencetropicalizes} becomes
\begin{equation*}\begin{tikzcd}
\Lambda^\vee\otimes\G_m^{\an}\arrow[rr]\arrow[d,"\trop"'] && \Lambda^\vee\otimes\G_m^{\an}/N=\Pic_0(A)^{\an}\arrow[d,"\trop"]\\
\Lambda^{\vee}\otimes\R\arrow[rr] &&\Lambda^{\vee}\otimes\R/N
\end{tikzcd}\end{equation*}
and this is commutative by the construction of the tropicalization map for line bundles in Section \ref{subsec:tropicalizing line bundles} above. The general case $r\geq 1$ follows, since both $E(.)$ in the algebraic and the tropical realm and both tropicalization maps naturally respect symmetric powers.
\end{proof}




\bibliographystyle{amsalpha}
\bibliography{biblio}{}

\end{document}